\documentclass[a4paper]{amsart}
\usepackage[latin1]{inputenc}
\usepackage{amsfonts}
\usepackage{amsmath,latexsym,amssymb,amsfonts, amsthm}
\usepackage{esint}
\usepackage{cite}
\usepackage{graphicx}
\usepackage{amscd}
\usepackage{color}
\usepackage{bm}             
\usepackage{enumerate}
\usepackage[dvips]{epsfig}
\usepackage{psfrag}
\usepackage{epsfig}

\usepackage[
  hmarginratio={1:1},     
  vmarginratio={1:1},     
  textwidth=15cm,        
  textheight=21cm,
  heightrounded,          
]{geometry}

\usepackage{graphicx,color}
\usepackage[colorlinks]{hyperref}
\hypersetup{linkcolor=blue,citecolor=blue,filecolor=black,urlcolor=blue}

\newtheorem{theorem}{Theorem}
\newtheorem{corollary}[theorem]{Corollary}
\newtheorem{proposition}[theorem]{Proposition}
\newtheorem{lemma}[theorem]{Lemma}
\theoremstyle{definition}
\newtheorem{remark}{Remark}

\newtheorem{definition}{Definition}

\numberwithin{theorem}{section}

\numberwithin{remark}{section}

\numberwithin{equation}{section}

\newcommand{\R}{\mathbb{R}}

\newcommand{\pa}{\partial}

\newcommand{\C}{{\mathbb C}}

\newcommand{\cP}{{\mathcal P}}

\newcommand{\dist}{{\rm dist}}
\newcommand{\supp}{{\rm supp}}

\newcommand{\weak}{\rightharpoonup}

\newcommand{\eps}{\varepsilon}

\DeclareMathOperator{\rad}{rad}

\title[Normalized ground states for the NLSE with combined nonlinearities]{Normalized ground states for the NLS equation with combined nonlinearities}

\author[N. Soave]{Nicola Soave}\thanks{}
\address{Nicola Soave \newline \indent
Dipartimento di Matematica,  Politecnico di Milano,  \newline \indent
Via Edoardo Bonardi 9, 20133 Milano, Italy}
\email{nicola.soave@gmail.com; nicola.soave@polimi.it}

\thanks{The author is partially supported by the ERC Advanced Grant 2013 n. 339958 ``Complex Patterns for Strongly Interacting Dynamical Systems - COMPAT'', by the PRIN-2015KB9WPT\texttt{\char`_}010 Grant: ``Variational methods, with applications to problems in mathematical physics and geometry", and by the INDAM-GNAMPA group.}

\begin{document}

\begin{abstract}
We study existence and properties of ground states for the nonlinear Schr\"odinger equation with combined power nonlinearities
\[
-\Delta u= \lambda u + \mu |u|^{q-2} u + |u|^{p-2} u \qquad \text{in $\R^N$, $N \ge 1$,}
\]
having prescribed mass
\[
\int_{\R^N} |u|^2 = a^2.
\]
Under different assumptions on $q<p$, $a>0$ and $\mu \in \R$ we prove several existence and stability/instability results. In particular, we consider cases when 
\[
2<q \le 2+ \frac{4}{N} \le p<2^*, \quad q \neq p,
\]
i.e. the two nonlinearities have different character with respect to the $L^2$-critical exponent. These cases present substantial differences with respect to purely subcritical or supercritical situations, which were already studied in the literature.

We also give new criteria for global existence and finite time blow-up in the associated dispersive equation.
\end{abstract}

\maketitle

\section{Introduction}

Starting from the seminal contribution by T. Tao, M. Visan and X. Zhang \cite{TaoVisZha},
the nonlinear Schr\"odinger equation with combined power nonlinearities
\begin{equation}\label{com nls}
i \psi_t + \Delta \psi + |\psi|^{p-2} \psi + \mu |\psi|^{q-2} \psi = 0 \qquad \text{in $\R^N$}
\end{equation}
attracted much attention. According to \cite{Caz, TaoVisZha}, the Cauchy problem for \eqref{com nls} is locally well posed, and the unique local solution has conservation of \emph{energy} 
\begin{equation}\label{def E}
E_\mu: H^1(\R^N, \C) \to \R, \quad E_\mu(u) = \int_{\R^N} \left( \frac{1}{2} |\nabla u|^2 -\frac{1}{p} |u|^{p} - \frac{\mu}{q} |u|^q \right)
\end{equation}
and of \emph{mass} 
\[
|u|_2^2 := \int_{\R^N} |u|^2.
\]
Global well-posedness, scattering, the occurrence of blow-up and more in general dynamical properties has been studied in \cite{TaoVisZha} and many papers \cite{AkaIbrKikNaw, ChMiZh, Feng, FukOht, GuZu, KiOhPoVi, LeCozMaRa, MiaXuZha, MiaZhaZhe, Zha} (see also the references therein). In this paper we study existence and properties of ground states with prescribed mass, with particular emphasis to the role played by the lower order term $\mu |\psi|^{q-2} \psi$ in comparison with the unperturbed case $\mu=0$, and to the relation between the different exponents $2<q<p<2^*$. Here and in what follows $2^*$ denotes the critical exponent for the Sobolev embedding $H^1(\R^N) \hookrightarrow L^p(\R^N)$ (that is, $2^* = 2N/(N-2)$ if $N \ge 3$, and $2^*=+\infty$ if $N=1,2$), and, since $q<p<2^*$, we always work in a subcritical framework. We point out that the critical case $p=2^*$ is of interest as well, but, requiring ad hoc techniques, is treated in the companion paper \cite{So}. 

To find stationary states, one makes the ansatz $\psi(t,x) = e^{-i \lambda t} u(x)$, where $\lambda \in \R$ is the chemical potential and $u: \R^N \to \C$ is a time-independent function. This ansatz yields
\begin{equation}\label{stat com}
-\Delta u  = \lambda u + |u|^{p-2} u + \mu |u|^{q-2} u \qquad \text{in $\R^N$}.
\end{equation}
A possible choice is then to fix $\lambda \in \R$, and to search for solutions to \eqref{stat com} as critical points of the \emph{action functional}
\[
\mathcal{A}(u):= \int_{\R^N} \left(\frac{1}{2} |\nabla u|^2 -\frac{\lambda}2 |u|^2 - \frac{\mu}{q} |u|^q - \frac{1}{p} |u|^p\right);
\]
in this case particular attention is devoted to \emph{least action solutions}, namely solutions minimizing $\mathcal{A}$ among all non-trivial solutions. 

Alternatively, one can search for solutions to \eqref{stat com} having prescribed mass, and in this case $\lambda \in \R$ is part of the unknown. This approach seems particularly meaningful from the physical point of view, since, in addition to being a conserved quantity for the time dependent equation \eqref{com nls}, the mass has often a clear physical meaning; for instance, it represents the power supply in nonlinear optics, or the total number of atoms in Bose-Einstein condensation, two main fields of application of the NLS. Moreover, this approach turns out to be useful also from the purely mathematical perspective, since it gives a better insight of the properties of the stationary solutions for \eqref{com nls}, such as stability or instability (this was already evident in the seminal contributions by H. Berestycki and T. Cazenave \cite{BerCaz}, and by T. Cazenave and P.-L. Lions \cite{CazLio}). For these reasons, here we focus on existence and properties of solutions to \eqref{stat com} with prescribed mass, a problem which was, up to now, essentially unexplored.

The existence of normalized stationary states can be formulated as the following problem: given $a>0$, $\mu \in \R$, and $2<q<p<2^*$, we aim to find $(\lambda,u) \in \R \times H^1(\R^N, \C)$ solving \eqref{stat com} together with the normalization condition 
\begin{equation}\label{norm}
|u|_2^2 = \int_{\R^N} |u|^2 = a^2.
\end{equation} 
Solutions can be obtained as critical points of the energy functional $E_\mu$ (defined in \eqref{def E}) under the constraint
\[
u \in S_a:= \left\{u \in H^1(\R^N,\C): \int_{\R^N} |u|^2 = a^2 \right\}.
\]
As $2<q<p<2^*$, it is standard that $E_\mu$ is of class $C^1$ in $H^1(\R^N,\C)$, and any critical point $u$ of $E_\mu|_{S_a}$ corresponds to a solution to \eqref{stat com} satisfying \eqref{norm}, with the parameter $\lambda \in \R$ appearing as Lagrange multiplier. We will be particularly interested in ground state solutions, defined as follows:

\begin{definition}
We say that $\tilde u$ is a \emph{ground state} of \eqref{stat com} on $S_a$ if it is a solution to \eqref{stat com} having minimal energy among all the solutions which belongs to $S_a$: 
\[
d E_\mu|_{S_a}(\tilde u) = 0 \quad \text{and} \quad E_\mu(\tilde u) = \inf\{E_\mu(u): \  d E_\mu|_{S_a}(u) = 0, \quad \text{and} \quad u \in S_a\}.
\]
The set of the ground states will be denoted by $Z_{a,\mu}$.
\end{definition}
If $E_\mu$ admits a global minimizer, then this definition naturally extends the notion of ground states from linear quantum mechanics; moreover, it allows to deal with cases when $E_\mu$ is unbounded from below on $S_a$. We also recall the notion of stability and instability we will be interested in:

\begin{definition}
$Z_{a,\mu}$ is \emph{orbitally stable} if for every $\eps>0$ there exists $\delta>0$ such that, for any $\psi_0 \in H$ with $\inf_{v \in Z_{a,\mu}} \|\psi_0 - v\|_H < \delta$, we have
\[
\inf_{v \in Z_{a,\mu}} \|\psi(t,\cdot) - v\|_H < \eps \qquad \forall t>0,
\]
where $\psi(t, \cdot)$ denotes the solution to \eqref{com nls} with initial datum $\psi_0$. \\
A standing wave $e^{-i \lambda t} u$ is \emph{strongly unstable} if for every $\eps>0$ there exists $\psi_0 \in H^1(\R^N,\C)$ such that $\|u-\psi_0\|_H <\eps$, and $\psi(t, \cdot)$ blows-up in finite time.
\end{definition}
We observe that the definition of stability implicitly requires that \eqref{com nls} has a unique global solution, at least for initial data $\psi_0$ sufficiently close to $Z_{a,\mu}$. 

\medskip

As we will see, existence and properties of ground states \eqref{stat com}-\eqref{norm} are strongly affected by further assumptions on the exponents and on the data. As far as we know, so far these issues were only studied assuming $2<q<p<2+4/N$, or $2+4/N<q<p<2^*$. It is well known that, when dealing with the Schr\"odinger equation, \emph{the $L^2$-critical exponent} 
\[
\bar p:= 2+4/N
\]
plays a special role. This is the threshold exponent for many dynamical properties such as global existence vs. blow-up, and the stability or instability of ground states. From the variational point of view, if the problem is purely $L^2$-subcritical, i.e. $2<q<p<\bar p$, then $E_\mu$ is bounded from below on $S_a$. Thus, for every $a, \mu >0$ a ground states can be found as global minimizers of $E_\mu|_{S_a}$, see \cite{Stu1} or \cite{Lions2,Shi}. Moreover, the set of ground states is orbitally stable \cite{CazLio, Shi}. In the purely $L^2$-supercritical case, i.e. $\bar p<q<p<2^*$, on the contrary, $E_\mu|_{S_a}$ is unbounded from below; however, exploiting the mountain pass lemma and a smart compactness argument, L. Jeanjean \cite{Jea} could show that a normalized ground state does exist for every $a, \mu>0$ also in this case. The associated standing wave is strongly unstable \cite{BerCaz, LeCoz}, due to the supercritical character of the equation. We point out that, in \cite{Jea, LeCoz, Lions2, Shi, Stu1}, more general nonlinearities are considered.

In what follows we carefully analyze the cases when the combined power nonlinearities in \eqref{com nls} are of mixed type, that is
\[
2<q \le \bar p \le p<2^*, \quad \text{with $p \neq q$ and $\mu \in \R$}.
\]
As we will see, the interplay between subcritical, critical and supercritical nonlinearities has a deep impact on the geometry of the functional and on the existence and properties of ground states. From some point of view, this can be considered as a kind of Brezis-Nirenberg problem in the context of normalized solutions: we have a homogeneous problem for which the structure of the ground states is known, and we analyze how the introduction of a lower order term modifies this structure. In this perspective, we think that it is natural to treat the coefficient $\mu$ in front of $|u|^{q-2} u$ as a parameter, fixing the coefficient of $|u|^{p-2} u$ in \eqref{stat com} to be $1$. Notice however that, by scaling, it is possible to reverse this choice when $\mu>0$. Notice also that, since the coefficient of the $|u|^{p-2}u$ is positive, we always consider a focusing ``leading" nonlinearity, while we allow both focusing ($\mu>0$) and defocusing ($\mu<0$) lower order term $|\psi|^{q-2} \psi$. 

It is worth to remark that, if we fix $\lambda$, then existence and variational characterization of least action solutions do not change for any choice $2<q < p <2^*$. Indeed, for every $\lambda<0$ and $\mu>0$ equation \eqref{stat com} has a least action solution (with positive action) which can be obtained minimizing $\mathcal{A}$ on the associated Nehari manifold, or by means of other variational principle (this is known since the classical paper \cite{BerLio} by H. Berestycki and P.-L. Lions). The number of positive real valued solutions, on the other hand, is affected by the choice of $q$ and $p$, see \cite{DaDelPiGue}. 

\medskip

For quite a long time the paper \cite{Jea} was the only one dealing with existence of normalized solutions in cases when the energy is unbounded from below on the $L^2$-constraint. More recently, however, problems of this type received much attention, see \cite{AcWe, BaDeV, BeJe, BeJeLu, BoCaGoJe, BuEsSe, JeLuWa} for normalized solutions to scalar equations in the whole space $\R^N$, \cite{BaSo, BaSo2, BaJe, BaJeSo, GoJe} for normalized solutions to systems in $\R^N$, and \cite{FibMer, NoTaVe1, NoTaVe2, NoTaVe3, PiVe} for normalized solutions to equations or systems in bounded domains\footnote{It is remarkable that, dealing with normalized solutions, problems in unbounded domains and in the whole space $\R^N$ have to be treated with completely different methods (this is often not the case if one fixes the Lagrange multiplier $\lambda$ in \eqref{stat com}, neglects the mass constraint, and works in a radial setting).}. Among the other contributions, we refer in particular to \cite{BaSo, BeJe, Jea}, which strongly inspired many techniques used here.

\subsection{Main results}

%
%
%


The first and simplest case to analyze is given by the choice $p=\bar p=2+4/N$, that is, the leading nonlinearity is $L^2$-critical and we have a $L^2$-subcritical lower order term. Denoting by $\bar a_N$ the critical mass for the $L^2$-critical Schr\"odinger equation (see Section \ref{sec: pre}), we have:

\begin{theorem}\label{thm: subcrit}
Let $N \ge 1$, $2<q<p=\bar p$. It results that:
\begin{itemize}
\item[$i$)] if $0<a<\bar a_N$, then:
\begin{itemize} 
\item[a)] for every $\mu>0$
\[
m(a,\mu):= \inf_{S_a} E_\mu < 0, 
\]
and the infimum is achieved by $\tilde u \in S_a$ with the following properties: $\tilde u$ is a real-valued positive function in $\R^N$, is radially symmetric, solves \eqref{stat com} for some $\tilde \lambda<0$, and is a ground state of \eqref{stat com}-\eqref{norm}.
\item[b)] for every $\mu<0$ it results 
\[
\inf_{S_a} E_\mu = 0, \quad \text{and problem \eqref{stat com}-\eqref{norm} has no solution at all}.
\]
\end{itemize}
\item[$ii$)] if $a = \bar a_N$, then:
\begin{itemize} 
\item[a)] for every $\mu>0$ it results 
\[
\inf_{S_a} E_\mu = -\infty.
\]
\item[b)] for every $\mu<0$ it results 
\[
\inf_{S_a} E_\mu=0, \quad \text{and problem \eqref{stat com}-\eqref{norm} has no solution at all}.
\]
\end{itemize}
\item[$iii$)] if $a > \bar a_N$, then for every $\mu \in \R$ it results 
\[
\inf_{S_a} E_\mu = -\infty.
\]
\end{itemize}
\end{theorem}

Now, in case ($i$-$a$), the set $Z_{a,\mu}$ of ground states is not empty. 

\begin{theorem}\label{thm: Z subcr}
If $0<a<\bar a_N$ with $\mu>0$, then
\[
Z_{a,\mu} = \left\{ e^{i \theta} |u| \ \text{for some $\theta \in \R$ and $|u|>0$ in $\R^N$} \right\},
\]
and the set $Z_{a,\mu}$ is orbitally stable. Moreover, if $\tilde u_\mu \in Z_{a,\mu}$, then $|\nabla \tilde u_\mu|_2 \to 0$ as $\mu \to 0^+$.
\end{theorem}

Here and in the rest of the paper $|\cdot|_2$ denotes the standard $L^2$-norm. The fact that for $0<a<\bar a_N$ we have that $|\nabla \tilde u_\mu|_2 \to 0$ as $\mu \to 0^+$ reflects the non-existence of positive normalized solutions on $S_a$ for the homogeneous $L^2$-critical equation (we refer again to Section \ref{sec: pre}).

The simple proofs of Theorem \ref{thm: subcrit}-\ref{thm: Z subcr} relies on the Pohozaev identity, on the adaptation of the Lions' concentration-compactness principle \cite{Lions1,Lions2}, and on the classical Cazenave-Lions' stability argument \cite{CazLio}, further developed in \cite{HajStu}. It is an open question whether problem \eqref{stat com}-\eqref{norm} admits solution in cases ($ii$-a) and ($iii$).

We mention that the existence of a positive radial ground state in case ($i$)-$a$) for the choice $\mu=1$ was proved in \cite{LeCozMaRa}; however, we will not only prove existence of a ground state, but also the relative compactness of all the minimizing sequences for $m(a,\mu)$. This seems to be new, and is essential for the stability. We further refer to \cite{LeCozMaRa}, and also to \cite{Feng, GuZu} for a discussion of global existence and finite time blow-up in this framework.

\medskip

We focus now on the more interesting case when $\bar p<p<2^*$, that is the leading term is $L^2$-supercritical and Sobolev subcritical. The energy functional $E_\mu$ in now unbounded both from above and from below on $S_a$, independently on $\mu \in \R$ and on $2<q\le \bar p$; however, the geometry of $E_\mu$ is strongly affected both by the sign of $\mu$, and by the exact choice of $q$. We discuss at first the case $2<q<\bar p$ with $\mu>0$. We use the notation
\begin{equation}\label{def gamma_p}
\gamma_p:= \frac{N(p-2)}{2p},
\end{equation}
and we denote by $C_{N,p}$ the best constant in the Gagliardo-Nirenberg inequality $H^1 \hookrightarrow L^p$ (see \eqref{GN ineq}).

\begin{theorem}\label{thm: subcr-supcr}
Let $N \ge 1$, $2<q<\bar p<p<2^*$, and let $a,\mu>0$. Let us also suppose that
\begin{equation}\label{cond sub sup}
\left(\mu a^{(1-\gamma_q) q}\right)^{\gamma_p p - 2} \left(a^{(1-\gamma_p) p} \right)^{2-\gamma_q q} \\
< \left( \frac{p(2-\gamma_q q)}{2C_{N,p}^p (\gamma_p p -\gamma_q q)} \right)^{2-\gamma_q q}\left( \frac{ q(\gamma_p p-2)}{2C_{N,q}^q (\gamma_p p- \gamma_q q)} \right)^{\gamma_p p-2}.
\end{equation}
Then the following holds:
\begin{itemize}
\item[$i$)] $E_\mu|_{S_a}$ has a critical point $\tilde u$ at negative level $m(a,\mu)<0$ which is an interior local minimizer of $E_\mu$ on the set 
\[
A_k := \left\{ u \in S_a: |\nabla u|_2^2 <k\right\},
\]
for a suitable $k>0$ small enough. Moreover, $\tilde u$ is a ground state of \eqref{stat com} on $S_a$, and any other ground state is a local minimizer of $E_\mu$ on $A_k$.
\item[$ii$)] $E_\mu|_{S_a}$ has a second critical point of mountain pass type $\hat u$ at level $\sigma(a,\mu)>m(a,\mu)$.
\item[$iii$)] Both $\tilde u$ and $\hat u$ are real-valued positive functions in $\R^N$, are radially symmetric, and solve \eqref{stat com} for suitable $\tilde \lambda, \hat \lambda<0$. Moreover, $\tilde u$ is also radially decreasing.
\end{itemize}
\end{theorem}

Regarding the stability:

\begin{theorem}\label{thm: Z supcr 1}
Let $N \ge 1$, $2<q<\bar p<p<2^*$, and let $a>0$. There exists $\tilde \mu>0$ sufficiently small such that, if $0<\mu< \tilde \mu$, then
\[
Z_{a,\mu} = \left\{ e^{i \theta} |u| \ \text{for some $\theta \in \R$ and $|u|>0$ in $\R^N$} \right\},
\]
and the set $Z_{a,\mu}$ is orbitally stable. On the contrary, for every $\mu>0$ satisfying \eqref{cond sub sup} the solitary wave $\psi(t,x) = e^{-i \hat{\lambda} t} \hat u(x)$  is strongly unstable.
\end{theorem}


\begin{remark}
Condition \eqref{cond sub sup} is not obtained by any limit process, and provide an explicit condition for $a$ and $\mu$ (which are not necessarily ``small"). In fact, we can take one between $a$ and $\mu$ as large as we want, provided that the other is sufficient small. Also $\tilde \mu$ in Theorem \ref{thm: Z supcr 1} does not come from a limit process, see Remark \ref{rem: on tilde mu}. 
\end{remark}

We recall that, in the unperturbed homogeneous case $\mu=0$, for any $a>0$ there exists a unique positive solution of the stationary equation, which gives rise to a ground state solution of the NLS equation with a positive energy $m(a,0)>0$, and the associated solitary wave is strongly unstable since we are in a $L^2$-supercritical regime (see e.g. \cite[Section 8]{Caz}). Therefore, Theorems \ref{thm: subcr-supcr} and \ref{thm: Z supcr 1} show that the introduction of a focusing ($\mu>0$) $L^2$-subcritical perturbation into a $L^2$-supercritical Schr\"odinger equation leads, on one side, to the stabilization of a system which was originally unstable; and, on the other side, it leads to the multiplicity of positive stationary solutions. From the variational point of view, the stabilization is reflected by the discontinuity of the ground state energy level $m(a,\mu)$: we have $m(a,\mu)<0$ for every $\mu>0$ not too large, while $m(a,0) >0$. A somehow similar picture was already observed, in a different model, in \cite[Theorem 1.6]{BeJe}, where the discontinuity was created by the introduction of a trapping potential. Regarding the multiplicity of positive normalized solutions, related results were established again in \cite[Theorem 1.6]{BeJe}, and in the main results of \cite{GoJe, JeLuWa}. We refer to Remark \ref{rem: con ccv} below for more details.

In view of the above discussion, it is natural to study the behavior of the ground states as $\mu \to 0^+$: 

\begin{theorem}\label{thm: mu to 0}
Let $a>0$. For sufficiently small $\mu>0$, let us denote by $\tilde u_\mu$ and $\hat u_\mu$ the positive solutions given by Theorem \ref{thm: subcr-supcr}. Then $m(a,\mu) \to 0^-$, and any ground state $\tilde u_\mu \in S_a$ for $E_\mu|_{S_a}$ satisfies $|\nabla \tilde u_\mu|_2 \to 0$ as $\mu \to 0^+$. Furthermore, $\sigma(a,\mu) \to m(a,0)$, and $\hat u_\mu \to \tilde u_0$ strongly in $H$ as $\mu \to 0^+$, where $\tilde u_0$ is the positive radial ground state of the homogeneous problem obtained for $\mu=0$.
\end{theorem}

The next result concerns existence of ground states when the lower order power becomes $L^2$-critical.

\begin{theorem}\label{thm: supcr2}
Let $N \ge 1$, $q=\bar p<p<2^*$, $a,\mu>0$. If
\begin{equation}\label{hp cr pos}
\mu a^{\frac{4}{N}} < \bar a_N^{\frac{4}{N}} =\frac{\bar p}{2 C_{N,\bar p}^{\bar p}},
\end{equation}
then $E_\mu|_{S_a}$ has a critical point $\tilde u$ at positive level $m(a,\mu)>0$, with the following properties: $\tilde u$ is a real-valued positive function in $\R^N$, is radially symmetric, solves \eqref{stat com} for some $\tilde \lambda<0$, and is a ground state of \eqref{stat com} on $S_a$.
\end{theorem}

\begin{remark}
The right hand side in \eqref{hp cr pos} is the limit, as $q \to \bar p^-$, of the right hand side in \eqref{cond sub sup}. For the equality in \eqref{hp cr pos}, we refer to Section \ref{sec: pre}. 
\end{remark}

From Theorems \ref{thm: subcr-supcr} and \ref{thm: supcr2} we deduce that there is a discontinuity in the ground state energy level also when $q$ reach $\bar p$ from below. In fact, the transition from the $L^2$-subcritical to the $L^2$-critical threshold drastically changes the geometry of $E_\mu|_{S_a}$, preventing the existence of a local minimizer in the latter case (no matter how small $\mu$ is). As a result, also the stability of ground states is lost.

\begin{theorem}\label{thm: Z supcr 2}
Under the assumptions of Theorem \ref{thm: supcr2}, we have that 
\[
Z_{a,\mu} = \left\{ e^{i \theta} |u| \ \text{for some $\theta \in \R$ and $|u|>0$ in $\R^N$} \right\};
\]
moreover, if $u$ is a ground state, then the associated Lagrange multiplier $\lambda$ is negative, and the standing wave $e^{-i\lambda t}u$ is strongly unstable.
\end{theorem}

Similarly to what we did in Theorem \ref{thm: mu to 0}, we can also study the behavior of ground states as $q \to \bar p^-$. This is the content of the next statement, where we denote by $m_q(a,\mu)$ and $\tilde u_q$ the ground state level and the ground state associated with a precise choice of $q$ in Theorem \ref{thm: subcr-supcr}.

\begin{theorem}\label{thm: q to bar p}
Let $a,\mu>0$ satisfy \eqref{hp cr pos}. Then, for any $q$ sufficiently close to $\bar p$ condition \eqref{cond sub sup} is satisfied, and we have: $m_q(a,\mu) \to 0^-$, and any ground state $\tilde u_q$ for $m_q(a,\mu)$ satisfy $|\nabla \tilde u_q|_2 \to 0$ as $q \to \bar p^-$.
\end{theorem}

We conjecture that $\sigma_q(a,\mu) \to m_{\bar p}(a,\mu)$, and that there is convergence of $\hat u_q$ towards a ground state for $m_{\bar p}(a,\mu)$. We decided to not insist on this point.

\medskip

We now turn to the case when $\mu <0$. Under this assumption the geometry of the functional does not change as $q$ passes from $L^2$-subcritical regime to the $L^2$-critical one. Therefore, we have a unified statement.

\begin{theorem}\label{thm: sup mu <0}
Let $N \ge 1$, $2<q \le \bar p <p<2^*$, $a>0$ and $\mu<0$. If
\begin{equation}\label{hp mu < 0 conv}
\left(|\mu| a^{q(1-\gamma_q)}\right)^{p \gamma_p -2} a^{p(1-\gamma_p)(2-q \gamma_q)} < \left(\frac{1- \gamma_p}{C_{N,q}^q(\gamma_p - \gamma_q)}\right)^{p \gamma_p -2} \left( \frac{1}{\gamma_p C_{N,p}^p}\right)^{2-q \gamma_q},
\end{equation}
then $E_\mu|_{S_a}$ has a critical point $\tilde u$ at positive level $m(a,\mu)>0$ with the following properties: $\tilde u$ is a real-valued positive function in $\R^N$, is radially symmetric, solves \eqref{stat com} for some $\tilde \lambda<0$, and is a ground state of \eqref{stat com} on $S_a$.
\end{theorem}

Moreover: 

\begin{theorem}\label{thm: Z supcr 3}
Under the assumptions of Theorem \ref{thm: sup mu <0}, we have that 
\[
Z_{a,\mu} = \left\{ e^{i \theta} |u| \ \text{for some $\theta \in \R$ and $|u|>0$ in $\R^N$} \right\};
\]
moreover, if $u$ is a ground state, then the associated Lagrange multiplier $\lambda$ is negative, and the standing wave $e^{-i\lambda t}u$ is strongly unstable.
\end{theorem}

\begin{remark}
Assumption \eqref{hp mu < 0 conv} looks similar to \eqref{cond sub sup} and \eqref{hp cr pos}. Nevertheless, they play very different roles. While \eqref{cond sub sup} and \eqref{hp cr pos} are used to describe the geometry of $E_\mu$ (and are not involved in compactness issues), assumption \eqref{hp mu < 0 conv} is fundamental in proving the convergence of Palais-Smale sequences when $\mu<0$ (and is not involved in the study of the geometry of $E_\mu|_{S_a}$). Under the assumptions of $q$ and $p$ covered by Theorems \ref{thm: subcr-supcr}, \ref{thm: supcr2} and \ref{thm: sup mu <0}, it is an interesting and difficult question to understand if a ground state solutions may exist without any assumption on $a$ and $\mu$. We believe that this is not the case. 
\end{remark}

\begin{remark}
We could study the behavior of the ground states as $\mu \to 0$ also in Theorems \ref{thm: supcr2} and \ref{thm: sup mu <0}. It is not difficult to modify the proof of Theorem \ref{thm: mu to 0} (convergence of $\hat u$) and deduce that in both Theorems \ref{thm: supcr2} and \ref{thm: sup mu <0} we have $m(a,\mu) \to m(a,0)$, and $\tilde u_\mu \to \tilde u_0$ strongly.
\end{remark}

In the proofs of Theorems \ref{thm: subcr-supcr}-\ref{thm: Z supcr 3}, a special role will be played by the \emph{Pohozaev set}
\begin{equation}\label{def P}
\cP_{a,\mu} = \left\{u \in S_{a}: P_\mu(u) = 0\right\},
\end{equation}
where
\begin{equation}\label{Poh}
P_\mu(u) := \int_{\R^N} |\nabla u|^2 - \gamma_p \int_{\R^N} |u|^p - \mu \gamma_q \int_{\R^N} |u|^q.
\end{equation}
It is well known that any critical point of $E_\mu|_{S_a}$ stays in $\cP_{a,\mu}$, as a consequence of the Pohozaev identity (we refer for instance to \cite[Lemma 2.7]{Jea}). Moreover, $\cP_{a,\mu}$ is \emph{a natural constraint}, in the following sense:
\begin{proposition}\label{prop: natural}
Suppose that either the assumptions of Theorem \ref{thm: subcr-supcr}, or those of Theorem \ref{thm: supcr2}, or else those of Theorem \ref{thm: sup mu <0} hold. Then $\cP_{a,\mu}$ is a smooth manifold of codimension $1$ in $S_a$. Moreover, if $u \in \cP_{a,\mu}$ is a critical point for $E_\mu|_{\cP_{a,\mu}}$, then $u$ is a critical point for $E_\mu|_{S_a}$.
\end{proposition}

The properties of $\cP_{a,\mu}$ are then intimately related to the minimax structure of $E_\mu|_{S_a}$, and in particular to the behavior of $E_\mu$ with respect to dilations preserving the $L^2$-norm. To be more precise, for $u \in S_a$ and $s \in \R$, let
\begin{equation}\label{def star}
(s \star u)(x) := e^{\frac{N}{2}s} u(e^s x), \qquad \text{for a.e. $x \in \R^N$}.
\end{equation}
It results that $s \star u \in S_a$, and hence it is natural to study the \emph{fiber maps}
\begin{equation}\label{def Psi intro}
\Psi^\mu_{u}(s) := E_\mu(s \star u) = \frac{e^{2s}}{2} \int_{\R^N} |\nabla u|^2 - \frac{e^{p \gamma_p  s}}{p} \int_{\R^N} |u|^p -  \mu \frac{e^{q \gamma_q  s}}{q} \int_{\R^N} |u|^q.
\end{equation}
We shall see that critical points of $\Psi^\mu_u$ allow to project a function on $\cP_{a,\mu}$. Thus, monotonicity and convexity properties of $\Psi^\mu_u$ strongly affect the structure of $\cP_{a,\mu}$ (and in turn the geometry of $E_\mu|_{S_a}$), and also have a strong impact on properties of the the time-dependent equation \eqref{com nls}.  

In this direction, let us consider the decomposition of $\cP$ into the disjoint union $\cP_{a,\mu} = \cP_+^{a,\mu} \cup \cP_0^{a,\mu} \cup \cP_-^{a,\mu}$, where
\begin{equation}\label{split P}
\begin{split}
&\mathcal{P}_+^{a,\mu}  := \left\{ u \in \cP_{a,\mu}: 2|\nabla u|_2^2 > \mu q \gamma_q^2 |u|_q^q + p \gamma_p^2 |u|_p^p \right\} = \left\{ u \in \cP_{a,\mu}: \ (\Psi_u^\mu)''(0) > 0 \right\}
  \\
&\mathcal{P}_-^{a,\mu}  := \left\{ u \in \cP_{a,\mu}: 2|\nabla u|_2^2 < \mu q \gamma_q^2 |u|_q^q + p \gamma_p^2 |u|_p^p \right\} =  \left\{ u \in \cP_{a,\mu}: \ (\Psi_u^\mu)''(0) < 0 \right\} \\
&\mathcal{P}_0^{a,\mu}  := \left\{ u \in \cP_{a,\mu}: 2|\nabla u|_2^2 = \mu q \gamma_q^2 |u|_q^q + p \gamma_p^2 |u|_p^p \right\} =  \left\{ u \in \cP_{a,\mu}: \ (\Psi_u^\mu)''(0) = 0 \right\}.
\end{split}
\end{equation}
Denoting by $S_{a,r}$ the subset of the radially symmetric functions in $S_a$, we have:

\begin{proposition}\label{prop: struct P}
1) Under the assumptions of Theorem \ref{thm: subcr-supcr}, we have $\cP_0^{a,\mu}= \emptyset$, both $\cP_+^{a,\mu}$ and $\cP_-^{a,\mu}$ are not empty, and 
\[
m(a,\mu) = \min_{\cP_+^{a,\mu}} E_\mu \quad {while} \quad \sigma(a,\mu) = \min_{\cP_-^{a,\mu} \cap S_{a,r}} E_\mu.
\]
2) Under the assumptions of both Theorems \ref{thm: supcr2} and \ref{thm: sup mu <0}, we have $\cP_+^{a,\mu} = \cP_0^{a,\mu}= \emptyset$, and 
\[
m(a,\mu) = \min_{\cP_-^{a,\mu}} E_\mu = \min_{\cP_-^{a,\mu} \cap S_{a,r}} E_\mu.
\]
\end{proposition}

\begin{remark}
By \cite[Lemma 2.9]{Jea}, the situation described in point 2) also takes place when $\bar p < p <2^*$ and $\mu=0$. For $2<q<\bar p$, Proposition \ref{prop: struct P} gives another explanation of the discontinuity of the ground state level $m(a,\mu)$ when $\mu \to 0^+$: for $\mu>0$ we have a splitting $\cP_{a,\mu}= \cP_-^{a,\mu} \cup \cP_+^{a,\mu}$ into two disjoint components, and the ground state level is achieved on $\cP_+^{a,\mu}$; as $\mu \to 0$, however, $\cP_+^{a,\mu}$ becomes empty, while we have convergence both of the levels $\min_{\cP_-^{a,\mu} \cap S_{a,r}} E_\mu$ to $m(a,0)$, and of the associated minimizers, see Theorem \ref{thm: mu to 0}.

In point 1), it is natural to expect that $\hat u$ is in fact a minimizer on $\cP_-^{a,\mu}$, and not only in $S_{a,r} \cap \cP_-^{a,\mu}$. \end{remark}

\begin{remark}\label{rem: con ccv}
The change of the topology in $\cP_{a,\mu}$ obtained by the introduction of a focusing $L^2$-subcritical perturbation is reminiscent to what happens to the Nehari manifold in inhomogeneous elliptic problems \cite{Tar}, or in elliptic problems with concave-convex nonlinearities \cite{AmbBreCer, GarPer}. This is somehow surprising, since in \eqref{stat com} all the power-nonlinearities are super-linear; the phenomenon is a direct consequence of the $L^2$-constraint $S_a$, and of the behavior of $E_\mu$ with respect to $L^2$-norm-preserving dilations. Similar ``concave" effects in superlinear problems with $L^2$-constraint were already observed in \cite{BeJe, GoJe, JeLuWa}, and are the source of the multiplicity of positive normalized solutions. 
\end{remark}

The analysis of $\Psi^\mu_u$ for $u \in S_a$ is not only fruitful in the description of the geometry of $E_\mu|_{S_a}$, but also allows to give a quite precise characterization of global existence vs. finite-time blow-up. These issues were firstly studied in \cite{TaoVisZha} where, for $L^2$-supercritical and focusing leading nonlinearities, the occurrence of finite-time blow-up was proved under assumptions on the weighted mass current and on mass and energy of the initial datum\footnote{For the precise assumptions, we refer to \cite[Theorem 1.5]{TaoVisZha}. We remark that, with the notations in \cite{TaoVisZha}, the case of a $L^2$-supercritical and focusing leading term corresponds to $\lambda_2<0$, with $4/N< p_2< 4/(N - 2)$.}. In a different (and complementary) perspective, we have the following results where, in addition to finite time blow-up, we also provide conditions for global existence.

\begin{theorem}\label{thm: gwp}
Let us assume that the assumptions of either Theorem \ref{thm: subcr-supcr} or Theorem \ref{thm: supcr2}, or else Theorem \ref{thm: sup mu <0} are satisfied. Let $u \in S_a$ be such that $E_\mu(u) < \inf_{\cP_-^{a,\mu}} E_\mu$. Then $\Psi^\mu_u$ has a unique global maximum point $t_{u,\mu}$, and:
\begin{itemize}
\item[1)] if $t_{u,\mu} >0$, then the solution $\psi$ of \eqref{com nls} with initial datum $u$ exists globally in time.
\item[2)] if $t_{u,\mu}<0$ and $|x| u \in L^2(\R^N,\C)$, then the solution $\psi$ of \eqref{com nls} with initial datum $u$ blows-up in finite time.
\end{itemize}
\end{theorem}

The theorem permits to reduce the discussion of global existence vs. finite time blow-up to the study of the $1$-variable function $\Psi_u^\mu$. The properties of $\Psi_u^\mu$ will be described in Lemmas \ref{lem: prop fiber sub-sup}, \ref{lem: fiber cr} and \ref{lem: fiber mu < 0}. Some immediate consequences are collected in the following corollary.

\begin{corollary}\label{cor: gwp}
For $u \in S_a$, let $\psi_u$ be the solution to \eqref{com nls} with initial datum $u$. We have:
\begin{itemize}
\item[1)] Under the assumptions of Theorems \ref{thm: subcr-supcr}, \ref{thm: supcr2} or Theorem \ref{thm: sup mu <0}, for every $u \in S_a$ there exist $s_1 \le s_2$ such that
\[
\begin{cases}
s < s_1 &\implies \quad \text{$\psi_{s\star u}$ is globally defined} \\ s> s_2 \text{ and }  |x| u \in L^2 & \implies \quad \text{$\psi_{s \star u}$ blows-up in finite time}.
\end{cases}
\]
\item[2)] Under the assumptions of Theorems \ref{thm: subcr-supcr}, \ref{thm: supcr2} or Theorem \ref{thm: sup mu <0}, if $P_\mu(u) >0$ and $E_\mu(u) < \inf_{\cP_-^{a,\mu}} E_\mu$, then $\psi_u$ is globally defined.
\item[3)] Under the assumptions of Theorems \ref{thm: subcr-supcr}, \ref{thm: supcr2} or Theorem \ref{thm: sup mu <0}, if $|\nabla u|_2$ is sufficiently small, then $\psi_u$ is globally defined.
\item[4)] Under the assumptions of Theorems \ref{thm: supcr2} or Theorem \ref{thm: sup mu <0}, if $|x| u \in L^2(\R^N,\C)$, $E_\mu(u) <  \inf_{\cP_-^{a,\mu}} E_\mu$, and $P_\mu(u)<0$, then $\psi_u$ blows-up in finite time.
\item[5)] Under the assumptions of Theorem \ref{thm: subcr-supcr}, if $|x| u \in L^2(\R^N,\C)$ and $E_\mu(u) < m(a,\mu)$, then $\psi_u$ blows-up in finite time.
\end{itemize}
\end{corollary}

\begin{remark}
Differently to what happen in \cite{TaoVisZha}, we don't make any assumption on the weighted mass current of the initial datum in order to prove finite time blow-up. Moreover, Theorem \ref{thm: gwp} yields blow-up for positive energy solutions (while in \cite{TaoVisZha} only negative energy solutions are considered). The price to pay is that we have to impose some conditions on $a$ and $\mu$.

The difference between the case $q<\bar p < p$ and $\mu>0$ with the others in Corollary \ref{cor: gwp} is motivated by the different properties of the fiber maps $\Psi_u^\mu$, see Lemmas \ref{lem: prop fiber sub-sup}, \ref{lem: fiber cr} and \ref{lem: fiber mu < 0}.
\end{remark}

In the rest of the paper we give the proofs of the main results. After having discussed some preliminaries in Section \ref{sec: pre}, we prove Theorems \ref{thm: subcrit} and \ref{thm: Z subcr} in Section \ref{sec: cr subcr}. In Section \ref{sec: compactness}, we discuss the compactness of Palais-Smale sequences in $L^2$-supercritical framework. It is worth to remark that, dealing with normalized solutions, the compactness is a highly non-trivial problem, even if we are in a Sobolev subcritical framework. In Sections \ref{sec: subcr-supcr}, \ref{sec: cr} and \ref{sec: mu<0} we focus on existence of ground states, proving Theorems \ref{thm: subcr-supcr}, \ref{thm: supcr2} and \ref{thm: sup mu <0} respectively. In doing this, we also prove Proposition \ref{prop: struct P}. At this point we focus on the properties of ground states, with particular emphasis to stability and instability. In Section \ref{sec: prop I} we prove Theorems \ref{thm: Z supcr 1}, \ref{thm: mu to 0} and \ref{thm: q to bar p}, and in Section \ref{sec:10} we prove Theorems \ref{thm: Z supcr 2} and \ref{thm: Z supcr 3} and Proposition \ref{prop: natural}. Finally, Theorem \ref{thm: gwp} and Corollary \ref{cor: gwp} on global existence and finite time blow-up are discussed in Section \ref{sec: gwp}.


\medskip

Regarding the notation, in this paper we deal with both complex and real-valued functions, which will be in both cases denoted by $u, v, \dots$. This should not be a source of misunderstanding. The symbol $\bar u$ will always be used for the complex conjugate of $u$. For $p \ge 1$, the (standard) $L^p$-norm of $u \in L^p(\R^N,\C)$ (or of $u \in L^p(\R^N,\R)$) is denoted by $|u|_p$. We simply write $H$ for $H^1(\R^N,\C)$, and $H^1$ for the subspace of real valued functions $H^1(\R^N,\R)$. Similarly, $H^1_{\rad}$ denotes the subspace of functions in $H^1$ which are radially symmetric with respect to $0$, and $S_{a,r} = H^1_{\rad} \cap S_a$. The symbol $\|\cdot\|$ is used only for the norm in $H$ or $H^1$. Denoting by $^*$ the symmetric decreasing rearrangement of a $H^1$ function, we recall that, if $u \in H$, then $|u| \in H^1$, $|u|^* \in H^1_{\rad}$, with 
\[
|\nabla |u|^*|_2 \le |\nabla |u||_2 \le |\nabla u|_2
\]
(it is well known that the symmetric decreasing rearrangement decreases the $L^2$-norm of gradients; regarding the last inequality for complex valued functions, we refer to \cite[Proposition 2.2]{HajStu}). The symbol $\weak$ denotes weak convergence (typically in $H$ or $H^1$). Capital letters $C, C_1, C_2, \dots$ denote positive constant which may depend on $N$, $p$ and $q$ (but never on $a$ or $\mu$), whose precise value can change from line to line. We also mention that, within a section, after having fixed the parameters $a$ and $\mu$ we may choose to omit the dependence of $E_\mu$, $S_a$, $P_\mu$, $\cP_{a,\mu}$, \dots on these quantities, writing simply $E$, $S$, $P$, $\cP$, \dots. 

\medskip

\paragraph{\textbf{Acknowledgments}} We thank Prof. Dario Pierotti and Prof. Gianmaria Verzini for several fruitful discussions.

\section{Preliminaries}\label{sec: pre}

In this section we collect several results which will be often used throughout the rest of the paper.

\subsection*{Preliminaries on the homogeneous NLSE} We focus here on the case $\mu=0$, and in particular to existence and properties of ground states for 
\[
E_0(u) = \int_{\R^N} \left(\frac12 |\nabla u|^2 - \frac1p |u|^p\right)
\]
on $S_a$. Classically, the problem is equivalent to the search of real valued solutions to
\begin{equation}\label{mu = 0}
\begin{cases}
-\Delta u = \lambda u + u^{p-1} & \text{in $\R^N$} \\
u>0 & \text{in $\R^N$} \\
\int_{\R^N} u^2 = a^2, & u \in H^1(\R^N),
\end{cases}
\end{equation}
for some $\lambda< 0$. Thanks to the homogeneity of the nonlinear term, the problem is equivalent, by scaling, to 
\begin{equation}\label{fixed lambda}
-\Delta u + u = u^{p-1}, \quad u>0 \qquad \text{in $\R^N$}, \quad u \in H^1.
\end{equation}
It is well known \cite{Stra, Kwo} that, for $p \in (2,2^*)$, equation \eqref{fixed lambda} has a unique solution $w_{N,p}$, up to translations, and that $w_{N,p}$ is radially symmetric and radially decreasing with respect to a point. Moreover, if $p \ge 2^*$ there is no solution. It is not difficult to deduce that if $p \in (2,2^*) \setminus \{\bar p\}$, then \eqref{mu = 0} has a unique solution for any $a>0$, while if $p = \bar p= 2+4/N$, then \eqref{mu = 0} is solvable for the unique value $a = |w_{N,\bar p}|_2$, which from now on is denoted by $\bar a_N$. Moreover, for $a=\bar a_N$ problem \eqref{mu = 0} has infinitely many different radial ground states.

\subsection*{Gagliardo-Nirenberg inequality} We recall that, for every $N \ge 1$ and $p \in (2,2^*)$, there exists a constant $C_{N,p}$ depending on $N$ and on $p$ such that
\begin{equation}\label{GN ineq}
|u|_p \le C_{N,p} |\nabla u|_2^{\gamma_p} |u|_2^{1-\gamma_p} \qquad \forall u \in H,
\end{equation}
where $\gamma_p$ is defined by \eqref{def gamma_p}.
Weinstein \cite{Wei} proved that equality is achieved by $w_{N,p}$ (and by any of its rescaling). Moreover, he obtained the best constant $C_{N,p}$ in terms of the $L^2$-norm of (a scaling of) $w_{N,p}$. In the special case $p= \bar p$, formula (1.3) in \cite{Wei} allows to characterized the critical mass $\bar a_N$ as
\begin{equation}\label{best const}
\bar a_N = \left( \frac{\bar p}{2 C_{N, \bar p}^{\bar p}} \right)^\frac{N}{4}.
\end{equation}

\subsection*{Homogeneous NLSE from a variational perspective}

From the variational point of view, the transition through the $L^2$-critical exponent $\bar p$ can be easily explained. By \eqref{GN ineq}, we have that
\[
E_0(u) \ge \frac12|\nabla u|_2^2 - \frac{C_{N,p}^p}{p} a^{(1-\gamma_p) p} |\nabla u|_p^{\gamma_p p},
\]
with $\gamma_p$ defined by \eqref{def gamma_p}. Notice that
\[
\gamma_p p = \frac{N}2(p-2) \    \begin{cases} <2 & \text{if $2<p< \bar p$} \\ =2 & \text{if $p= \bar p$} \\ >2 & \text{if $\bar p<p < 2^*$}. \end{cases}
\]
This implies that $E_0$ is bounded from below on $S_a$ for $p< \bar p$ (for every choice of $a>0$), and for $p= \bar p$ provided that $a \le \bar a_N$. In the remaining cases, it is not difficult to check that $E_{0}|_{S_a}$ is unbounded from below: for $s \in \R$ and $u \in S_a$, we consider the scaling $s \star u$, defined in \eqref{def star},
and we observe that $s \star u \in S_a$ and 
\[
E_0(s \star u) = \frac{e^{2s}}{2} \int_{\R^N} |\nabla u|^2 - \frac{e^{\gamma_p p  s}}{p} \int_{\R^N} |u|^p.
\]
We deduce that, if $p> \bar p$ (so that $\gamma_p p>2$), then $E_0(s \star u) \to -\infty$ as $s \to +\infty$, for every $u \in S_a$, while in case $p= \bar p$ the same holds for all functions $u \in S_a$ with 
\[
\frac12 |\nabla u|_2^2 - \frac1{\bar p} |u|_{\bar p}^{\bar p}<0.
\]
Such a function does exist only if $a> \bar a_N$. 


\subsection*{Behavior of $E_\mu$ with respect to dilations.} A crucial role in the proof of all our results is represented by the study of the behavior of $E_\mu$ with respect to the $L^2$-norm preserving variations defined by \eqref{def star}. We consider, for $u \in S_a$ and $s \in \R$, the fiber $\Psi_u^\mu$ introduced in \eqref{def Psi intro}.
We have
\begin{equation*}
\begin{split}
(\Psi^\mu_u)'(s) &=  e^{2s}\int_{\R^N} |\nabla u|^2 - \gamma_p e^{p \gamma_p  s} \int_{\R^N} |u|^p -  \mu \gamma_q e^{q \gamma_q  s} \int_{\R^N} |u|^q \\
& = \int_{\R^N} |\nabla (s \star u)|^2 - \gamma_p \int_{\R^N} |s \star u|^p - \mu \gamma_q \int_{\R^N} |s \star u|^q = P_\mu(s \star u),
\end{split}
\end{equation*}
where $P_\mu$ is defined by \eqref{Poh}. Therefore:
\begin{proposition}\label{prop: psi P}
Let $u \in S_a$. Then: $s \in \R$ is a critical point for $\Psi^\mu_u$ if and only if $s \star u \in \cP_{a,\mu}$. 
\end{proposition}

In particular, $u \in \cP_{a,\mu}$ if and only if $0$ is a critical point of $\Psi_u^\mu$. For future convenience, we also recall that the map
\begin{equation}\label{cont star}
(s,u) \in \R \times H^1 \mapsto (s \star u) \in H^1 \quad \text{is continuous},
\end{equation}
see \cite[Lemma 3.5]{BaSo2}.

\section{$L^2$-critical leading term}\label{sec: cr subcr}

In this section we prove Theorem \ref{thm: subcrit}.
It is useful to observe that, in the present setting, \eqref{def Psi intro} reads
\begin{equation}\label{E star cr-subcr}
\begin{split}
E_\mu(s \star u) &= e^{2s} \left( \int_{\R^N}\frac{1}{2} |\nabla u|^2 - \frac1{\bar p} |u|^{\bar p} \right) - \mu \frac{e^{q\gamma_q s}}{q} \int_{\R^N} |u|^q  = e^{2s} E_0(u) - \mu \frac{e^{q\gamma_q s}}{q} \int_{\R^N} |u|^q.
\end{split}
\end{equation}

\medskip

\noindent \textbf{The case $0<a \le \bar a_N$ with $\mu < 0$.} If there exists a solution $u$ to \eqref{stat com}-\eqref{norm}, then by Pohozaev identity $P_\mu(u) = 0$, and hence
\[
\int_{\R^N} |\nabla u|^2 = \frac{2}{\bar p}\int_{\R^N} |u|^{\bar p} + \mu \gamma_q \int_{\R^N} |u|^q.
\]
As recalled in Section \ref{sec: pre}, we have that $\inf_{S_a} E_0 \ge 0$ since $a<\bar a_N$, and hence we deduce that 
\[
0> \mu \gamma_q \int_{\R^N} |u|^q = 2 E_0(u)  \ge 2 \inf_{S_a} E_0 \ge 0,
\]
a contradiction.    

\medskip

\noindent{\textbf{The case $a = \bar a_N$ with $\mu>0$}.} Since $a= \bar a_N$, there exists $w=w_{N, \bar p} \in S_{a}$ with $E_0(w) = 0$. Therefore, by \eqref{E star cr-subcr}, 
\[
E_\mu(s \star w) = - \mu  \frac{e^{q\gamma_q s}}{q} |w|_q^q \to -\infty \qquad \text{as $s \to +\infty$.}
\]

\medskip

\noindent{\textbf{The case $a > \bar a_N$}.} Since $a> \bar a_N$, there exists $u \in S_{a}$ with $E_0(u) <0$. Using \eqref{E star cr-subcr} and the fact that $2>q \gamma_q$, we deduce again that $\inf_{S_{a}} E_\mu = -\infty$.

\medskip

\noindent{\textbf{The case $a<\bar a_N$ with $\mu>0$.}} At first, we show that $E_\mu$ is bounded from below on $S_a$, and that the infimum is negative. By the Gagliardo-Nirenberg inequality
\begin{equation}\label{691}
E_\mu(u) \ge \frac12 \left(1 - \frac2{\bar p} C_{N, \bar p}^{\bar p} a^{\bar p-2}\right)  |\nabla u|_2^2 - \frac{\mu}q a^{q(1-\gamma_q)} |\nabla u|_2^{\gamma_q q},
\end{equation}
for every $u \in S_a$. Since $a<\bar a_N$, $\gamma_q q<2$, and since the coefficient of $|\nabla u|_2^2$ is positive by \eqref{best const}, we have that $E_\mu$ is coercive on $S_a$, and $m(a,\mu) := \inf_{S_a} E_\mu>-\infty$. The fact that $m(a,\mu)<0$ follows by \eqref{E star cr-subcr}, since being $\mu>0$ we have that $E_\mu(s \star u)<0$ for every $(s,u) \in \R \times S_a$ with $s \ll-1$. Furthermore, we observe that $\inf_{S_a \cap H^1} E_\mu = \inf_{S_a} E_\mu$, since if $u \in H$ we have that $|u| \in S_a \cap H^1$ and $|\nabla |u||_2 \le |\nabla u|_2$. Now:

\begin{proposition}\label{prop: rel cpt sub}
Let $\{u_n\} \subset H^1(\R^N,\R)$ be a sequence such that
\[
E_\mu(u_n) \to m(a,\mu), \quad \text{and} \quad |u_n|_2 \to a.
\]
Then $\{u_n\}$ is relatively compact in $H^1$ up to translations; that is, there exist a subsequence $\{u_{n_k}\}$, a sequence of points $\{y_k\} \subset \R^N$, and a function $\tilde u \in S_a \cap H^1$ such that $u_{n_k}(\cdot + y_k) \to \tilde u$ strongly in $H^1$.
\end{proposition}

Here we only consider real-valued functions. Indeed, using the argument developed in \cite[Section 3]{HajStu}, if relative compactness holds in $H^1(\R^N,\R)$, then one can easily deduce that it also holds in $H^1(\R^N,\C)$.

\begin{remark}
If one is only interested in the existence of a real-valued, positive and radial ground state, it is possible to work with a minimizing sequence of radially decreasing functions, and exploit their compactness properties. This approach was followed in \cite{LeCozMaRa}. However, the relative compactness of minimizing sequences is a stronger result which allows to prove the stability of the ground states set $Z_{a,\mu}$.
\end{remark}

The proof of Proposition \ref{prop: rel cpt sub} is an application of the concentration-compactness principle by P. L. Lions \cite{Lions1,Lions2}, and rests on the validity of the strict sub-additivity for $a \mapsto m(a,\mu)$.

\begin{lemma}\label{lem: sub sub}
Let $a_1, a_2>0$ be such that $a_1^2 + a_2^2 = a^2< \bar a_N^2$. Then
\[
m(a,\mu)< m(a_1,\mu) + m(a_2,\mu).
\]
\end{lemma} 
\begin{proof}
Let $0<c<\bar a_N$, let $\theta>1$ be such that $\theta c< \bar a_N$, and let $\{u_n\} \subset S_c$ be a minimizing sequence for $m(c,\mu)$. Then 
\begin{align*}
m(\theta c, \mu) &\le E_\mu(\theta u_n) = \frac12 \theta^2 |\nabla u_n|_2^2 - \frac{\mu  \theta^q}{q} |u_n|_q^q - \frac{\theta^p}{p} |u_n|_p^p  < \theta^2 E_\mu(u_n),
\end{align*}
since $\theta>1$ and $q,p>2$. As a consequence $m(\theta c,\mu) \le \theta^2 m(c,\mu)$, with equality if and only if $|u_n|_p^p + |u_n|_q^q \to 0$ as $n \to \infty$. But this is not possible, since otherwise we would find
\[
0> m(c,\mu) = \lim_{n \to \infty} E_\mu(u_n) \ge \liminf_{n \to \infty} \frac12 |\nabla u_n|_2^2 \ge 0,
\]
a contradiction. Thus, we have the strict inequality $m(\theta c,\mu) < \theta^2 m(c,\mu)$, and from this the thesis follows as in \cite[Lemma II.1]{Lions1}.
\end{proof}

\begin{proof}[Proof of Proposition \ref{prop: rel cpt sub}]
By \eqref{691}, and since $a_n \to a< \bar a_N$, the sequence $\{u_n\}$ is bounded in $H^1$. Thus, by the concentration-compactness principle (see in particular \cite[Lemma III.1]{Lions1}) applied to $v_n=a/a_n u_n$, there exists a subsequence, still denoted by $\{v_n\}$ satisfying one of the following three possibilities:\\
$i$) \emph{vanishing:} 
\[
\lim_{k \to \infty} \sup_{y \in \R^N} \int_{B_R(y)} |v_n|^2 =0 \qquad \forall R>0.
\]
$ii$) \emph{dichotomy:} there exists $a_1 \in (0,a)$ and $\{v_n^1\}$, $\{v_n^2\}$ bounded in $H^1$ such that as $n \to \infty$
\[
\begin{array}{l c l}
 |v_n - (v_n^1 + v_n^2)|_r \to 0 \qquad \text{for $2 \le r < 2^*$}; & & |v_n^1|_2 \to a_1 \quad \text{and} \quad |v_n^2|_2 \to \sqrt{a^2-a_1^2}; \\
 \dist(\supp \,v_n^1, \supp \, v_n^2) \to +\infty; & &\liminf_{n \to \infty} \Big[|\nabla v_n|_2^2 -|\nabla v_n^1|_2^2 -|\nabla v_n^2|_2^2 \Big] \ge 0.
\end{array}
\] 
$iii$) \emph{compactness:} there exists $y_n \in \R^N$ such that:
\[
\forall \eps>0 \quad \exists R>0: \quad \int_{B_R(y_k)} v_n^2 \ge a^2-\eps.
\]
Vanishing cannot occur, since otherwise $u_n \to 0$ strongly in $L^r(\R^N)$ for every $r \in (2,2^*)$ (see \cite[Lemma I.1]{Lions2}), whence it follows that $\liminf_n E_\mu(u_n) \ge 0$, in contradiction with $m(a,\mu)<0$. 

Also dichotomy cannot occur, since otherwise
\[
m(a,\mu) =\lim_{n \to \infty} E_\mu(u_n) = \lim_{n \to \infty} E_\mu(v_n) \ge \limsup_{n \to \infty} \left(E_\mu(v_n^1) + E_\mu(v_n^2) \right) \ge m(a_1,\mu) + m(a_2,\mu),
\]
in contradiction with Lemma \ref{lem: sub sub} (in the second equality, we used the facts that $\{u_n\}$ is bounded in $H^1$ and $a_n \to a$).

Therefore, compactness hold, and the sequence of translations $\tilde v_n := v_n(\cdot + y_n)$ converges, strongly in $L^2(\R^N)$ (and weakly in $H^1$), to a limit $\tilde u \in S_a \cap H^1$. Since $a_n \to a$ and $\{u_n\}$ is bounded, we deduce that in fact $\tilde u_n := u_n(\cdot + y_n)$ converges, strongly in $L^2(\R^N)$, to $\tilde u$. If $r \in (2,2^*)$, by H\"older and Sobolev inequality 
\[
|\tilde u_n - \tilde u|_r^r \le |\tilde u_n - \tilde u|_2^{2(1-\alpha)} |\tilde u_n - \tilde u|_{2^*}^{2^*\alpha} \le C |\tilde u_n - \tilde u|_2^{2\alpha} \to 0
\]
(for some $\alpha \in (0,1)$), whence
\[
m(a,\mu) \le E_\mu(\tilde u) \le \liminf_{n \to \infty} E_\mu(\tilde u_n) = \liminf_{n \to \infty} E_\mu(u_n) = m(a,\mu).
\]
We finally deduce that the previous inequalities are equalities, and in particular $\|\tilde u_n\| \to \|\tilde u\|$. This shows the relative compactness of any minimizing sequence for $m(a,\mu)$ of real valued functions, up to translations.
\end{proof}

We need two further ingredients in order to proceed with the stability.

\begin{lemma}\label{lem: cont gsl}
The function $a \in (0,\bar a_N)  \mapsto m(a,\mu)$ is continuous.
\end{lemma}
\begin{proof}
Let $a_n \to a  \in (0, \bar a_N)$. For every $n$ there exists $u_n \in S_{a_n}$ such that $m(a_n,\mu) \le E_\mu(u_n) < m(a_n,\mu) + 1/n$. By estimate \eqref{691}, taking into account that $a_n \le a +\eps < \bar a_N$ for $n$ sufficiently large (and $\eps>0$ sufficiently small), we deduce that $E_\mu|_{S_{a_n}}$ are equi-coercive, and hence $\{u_n\}$ is bounded in $H$. Now, let us consider $v_n:= a/a_n u_n \in S_a$. We have
\begin{align*}
m(a,\mu) & \le  E(v_n) = E(u_n) + \frac12\left( \frac{a^2}{a_n^2}-1\right) |\nabla u_n|_2^2 - \frac1p\left(\frac{a^p}{a_n^p}-1\right) |u_n|_p^p -  \frac{\mu}q\left(\frac{a^q}{a_n^q}-1\right) |u_n|_q^q \\
& = E(u_n) + o(1),
\end{align*}
where we used the boundedness of $\{u_n\}$ and the fact that $a_n \to a$. Passing to the limit as $n \to \infty$, we deduce that 
\[
m(a,\mu) \le \liminf_{n \to \infty} m(a_n,\mu).
\]

In a similar way, let $\{w_n\}$ be a minimizing sequence for $m(a,\mu)$, which is bounded by \eqref{691}, and let $z_n:= a_n/a w_n \in S_{a_n}$. Then we have
\[
m(a_n,\mu) \le E(z_n) = E(w_n) + o(1) \quad \implies \quad \limsup_{n \to \infty} m(a_n,\mu) \le m(a,\mu). \qedhere
\]
\end{proof}

\begin{lemma}\label{lem: gl ex 2910}
If $a \in (0,a_N)$ and $\mu>0$, then any solution $\psi$ to \eqref{com nls} with initial datum $u \in S_a$ is globally defined in time.
\end{lemma}

\begin{proof}
Denoting by $(-T_{\min}, T_{\max})$ the maximal existence interval for $\psi$, we have classically that either $\psi$ is globally defined for positive times, or $|\nabla \psi(t)|_2 = +\infty$ as $t \to T_{\max}^-$ (and an analogue alternative holds for negative times), see \cite[Section 3]{TaoVisZha}. Supposing that $T_{\max}<+\infty$, we have then that $|\nabla \psi(t)|_2 \to +\infty$ as $t \to T_{\max}^-$, and as a consequence $E_\mu(\psi(t)) \to +\infty$ as $t \to T_{\max}^-$, by \eqref{691}. This is in contradiction with the conservation of the energy.
\end{proof}

\begin{proof}[Conclusion of the proof of Theorem \ref{thm: subcrit}]
Proposition \ref{prop: rel cpt sub} immediately implies the existence of a real-valued minimizer $\tilde u$ for $E_\mu$ on $S_a \cap H^1$. Denoting by $|u|^*$ the Schwarz rearrangement of $|u| \in H^1$, we observe that, since $E_\mu(|u|^*) \le E_\mu(u)$ and $|u|^* \in S_a$, we can suppose that $u \ge 0$ is radially symmetric and decreasing. Being a critical point of $E_\mu$ on $S_a \cap H^1$, $u$ is a real-valued solution to \eqref{stat com}-\eqref{norm} for some $\tilde \lambda \in \R$, and by regularity it is of class $C^2$; the strong maximum principle yields $u>0$ in $\R^N$. Finally, multiplying \eqref{stat com} by $\tilde u$ and integrating, we obtain
\[
\tilde \lambda a^2 = |\nabla u|_2^2 - \mu |u|_q^q - |u|_p^p = 2 m(a,\mu) + \mu \left( \frac{2}{q}-1\right) |u|_q^q + \left( \frac2p-1\right)|u|_p^p < 0,
\]
which shows that $\tilde \lambda<0$.
\end{proof}

\begin{proof}[Proof of Theorem \ref{thm: Z subcr}]
The validity of Proposition \ref{prop: rel cpt sub} for complex valued function can be proved exactly as in Theorem 3.1 in \cite{HajStu}, starting from the same property for real-valued functions and using Lemma \ref{lem: cont gsl}. Thus, the orbital stability of $Z_{a,\mu}$ can be proved following the classical Cazenave-Lions argument \cite{CazLio}, using the relative compactness of minimizing sequences in $H$ up to translations, and the global existence result in Lemma \ref{lem: gl ex 2910}. The structure of the set $Z_{a,\mu}$ can be determined exactly as in Theorem 4.1 of \cite{HajStu}. Finally, the asymptotic behavior of the ground states as $\mu \to 0^+$ follows directly from \eqref{691}, since we have
\[
0 > E_\mu(\tilde u_\mu) \ge \frac12 \left(1 - \frac2{\bar p} C_{N, \bar p}^{\bar p} a^{\bar p-2}\right)  |\nabla \tilde u_\mu|_2^2 - \frac{\mu}q a^{q(1-\gamma_q)} |\nabla \tilde u_\mu|_2^{\gamma_q q},
\]
whence
\[
\frac12 \left(1 - \frac2{\bar p} C_{N, \bar p}^{\bar p} a^{\bar p-2}\right)  |\nabla \tilde u_\mu|_2^{2- \gamma_q q} < \frac{\mu}q a^{q(1-\gamma_q)} \to 0
\]
as $\mu \to 0^+$.
\end{proof}

\section{Compactness of Palais-Smale sequences in the $L^2$-supercritical setting}\label{sec: compactness}

When the exponent $p$ in \eqref{stat com} is $L^2$-supercritical, the compactness of a Palais-Smale sequence (we will often write PS sequence for short) is a highly nontrivial issue. The boundedness of a PS sequence is not guaranteed in general\footnote{With respect to problems without normalization condition, we observe that if $u \in S_a$, then $u \not \in T_u S_a$, and hence cannot be used as test function; the standard argument to prove boundedness of PS sequence in a Sobolev subcritical setting relies on this fact.}; also, sequences of approximated Lagrange multipliers have to be controlled; and moreover, weak limits of PS sequence could leave the constraint, since the embeddings $H^1(\R^N) \hookrightarrow L^2(\R^N)$ and also $H^1_{\rad}(\R^N) \hookrightarrow L^2(\R^N)$ are not compact. 

In what follows we discuss therefore the convergence of special PS sequences, satisfying suitable additional conditions, following the ideas firstly introduced by L. Jeanjean in \cite{Jea}. As a preliminary remark, we note that, since $E_\mu$ is invariant under rotations, critical points (resp. PS sequences) of $E_\mu$ restricted on $S_{a,r}$ are critical points (resp. PS sequences) of $E_\mu$ on $S_a$.


\begin{lemma}\label{lem: conv PS subcr}
Let $N \ge 2$, and $2 <q \le 2+4/N<p<2^*$. Let $\{u_n\} \subset S_{a,r}$ be a Palais-Smale sequence for $E_\mu|_{S_{a}}$ at level $c \neq 0$, and suppose in addition that: 
\begin{itemize}
\item[($i$)] $P_\mu(u_n) \to 0$ as $n \to \infty$.
\item[($ii$)] Either $\mu>0$ (without any additional assumption), or $\mu<0$ and \eqref{hp mu < 0 conv} holds.
\end{itemize}
Then up to a subsequence $u_n \to u$ strongly in $H^1$, and $u \in S_a$ is a real-valued radial solution to \eqref{stat com} for some $\lambda<0$.
\end{lemma}


\begin{proof}
The proof is divided into four main steps. 

\medskip

\noindent \textbf{Step 1) Boundedness of $\{u_n\}$ in $H^1$.} We consider at first the case $q=2+4/N= \bar p$, and we recall that with this choice $\gamma_{\bar p}=2/{\bar p}$. Then, as $P_\mu(u_n) \to 0$, we have
\begin{equation}\label{766}
|\nabla u_n|_2^2 = \mu \frac{2}{\bar p} |u_n|_{\bar p}^{\bar p} + \gamma_p |u_n|_p^p + o(1) \qquad \text{as $n \to \infty$}.
\end{equation}
Let us assume by contradiction that $|\nabla u_n|_2 \to +\infty$. Thus, by \eqref{766} we deduce that
\[
\frac1p\left(\frac{\gamma_p p}{2}- 1 \right) |u_n|_p^p + o(1) = E_\mu(u_n) \le c+1, \quad \text{and} \quad \mu \frac{2}{\bar p} |u_n|_{\bar p}^{\bar p} + \gamma_p |u_n|_p^p \to +\infty,
\]
with $\gamma_p p>2$ since $p>\bar p$. This gives immediately a contradiction for $\mu<0$; if instead $\mu>0$, we infer that $\{|u_n|_p\}$ is bounded, with $|u_n|_{\bar p} \to +\infty$. On the other hand, by the H\"older inequality there exists $\alpha \in (0,1)$ (depending on $p$ and $N$) such that $|u_n|_{\bar p} \le |u_n|_p^\alpha |u_n|_2^{1-\alpha} \le C$, which gives the desired contradiction also for $\mu>0$.

Let now $2<q<\bar p$. As $P_\mu(u_n) \to 0$, we observe that
\[
|u_n|_p^p = \frac{1}{\gamma_p} |\nabla u_n|_2^2 - \mu \frac{\gamma_q}{\gamma_p} |u_n|_q^q + o(1),
\]
whence
\[
E_\mu(u_n) = \left( \frac{1}{2}- \frac{1}{\gamma_p p} \right) |\nabla u_n|_2^2 - \frac{\mu}{q} \left( 1- \frac{\gamma_q q}{\gamma_p p} \right) |u_n|_q^q + o(1),
\]
and both the coefficients inside the brackets are positive. Thus, if $\mu<0$ we immediately deduce that $\{u_n\}$ is bounded, while if $\mu>0$, by the Gagliardo-Nirenberg inequality we have that
\[
c+1  \ge E_\mu(u_n) \ge \left( \frac{1}{2}- \frac{1}{\gamma_p p} \right) |\nabla u_n|_2^2 - \frac{\mu}{q} \left( 1- \frac{\gamma_q q}{\gamma_p p} \right)  C_{N,q}^q a^{(1-\gamma_q)q} |\nabla u_n|_2^{\gamma_q q};
\]
this implies that
\[
|\nabla u_n|_2^2 \le C \mu a^{(1-\gamma_q)q} |\nabla u_n|_2^{\gamma_q q} + C,
\]
and, since $\gamma_q q <2$, the boundedness of $\{u_n\}$ follows also in this case. 

\medskip

\noindent \textbf{Step 2)} Since $N\ge 2$, the embedding $H^1_{\rad}(\R^N) \hookrightarrow L^r(\R^N)$ is compact for $r \in (2,2^*)$, and we deduce that there exists $u \in H^1_{\rad}$ such that, up to a subsequence, $u_n \weak u$ weakly in $H^1$, $u_n \to u$ strongly in $L^r(\R^N)$ for $r \in (2,2^*)$, and a.e. in $\R^N$. Now, since $\{u_n\}$ is a bounded Palais-Smale sequence of $E_\mu|_{S_a}$, by the Lagrange multipliers rule there exists $\lambda_n \in \R$ such that
\begin{equation}\label{768}
\textrm{Re} \int_{\R^N} \nabla u_n \cdot \nabla \bar \varphi - \lambda_n u_n \bar \varphi - \mu |u_n|^{q-2} u_n \bar \varphi - |u_n|^{p-2} u_n \bar \varphi = o(1) \|\varphi\|,
\end{equation}
for every $\varphi \in H$, where $o(1) \to 0$ as $n \to \infty$, and $\textrm{Re}$ stays for the real part. The choice $\varphi = u_n$ provides 
\[
\lambda_n a^2 = |\nabla u_n|_2^2 - \mu |u_n|_q^q - |u_n|_p^p + o(1),
\]
and the boundedness of $\{u_n\}$ in $H^1 \cap L^p \cap L^q$ implies that $\{\lambda_n\}$ is bounded as well; thus, up to a subsequence $\lambda_n \to \lambda \in \R$. 

\medskip

\noindent  \textbf{Step 3) \boldsymbol{$\lambda <0$}.} We consider separately $\mu>0$ and $\mu<0$, starting from the former one. Recalling that $P_\mu(u_n) \to 0$, we have
\begin{equation}\label{767}
\lambda_n a^2 = \mu(\gamma_q -1) |u_n|_q^q +(\gamma_p -1) |u_n|_p^p + o(1).
\end{equation}
Since $\mu>0$, $0<\gamma_q, \gamma_p<1$, we deduce that $\lambda\le 0$, with equality only if $u \equiv 0$. But $u$ cannot be identically $0$, since $E_\mu(u_n) \to c \neq 0$: indeed, using again the fact that $P_\mu(u_n) \to 0$, if we had $u_n \to 0$ we would find by strong $L^p$ and $L^q$ convergence that
\[
E_\mu(u_n) = \frac{\mu}{q}\left( \frac{\gamma_q q}{2}- 1 \right) |u_n|_q^q + \frac{1}{p} \left( \frac{\gamma_p p}{2}- 1 \right) |u_n|_p^p + o(1) \to 0,
\]
a contradiction. Coming back to \eqref{767}, we proved that up to a subsequence $\lambda_n \to \lambda<0$. 

The case $\mu<0$ is more involved. Since $P_\mu(u_n) \to 0$, we have that
\[
|\nabla u_n|_2^2 = \mu \gamma_q |u_n|_q^q + \gamma_p |u_n|_p^p + o(1) \le   \gamma_p |u_n|_p^p + o(1) \quad \implies \quad |\nabla u|_2^2 \le \gamma_p |u|_p^p.
\]
Then, by the Gagliardo-Nirenberg inequality,
\[
|\nabla u|_2^2 \le \gamma_p |u|_p^p \le \gamma_p C_{N,p}^p |u|_2^{p(1-\gamma_p)} |\nabla u|_2^{p \gamma_p}.
\]
As in the case $\mu>0$, we have $u \not \equiv 0$ since otherwise $E_\mu(u_n) \to 0$, in contradiction with the assumptions. Therefore, using that $|u|_2 \le a$ by weak lower semi-continuity, we deduce that
\begin{equation}\label{3081}
|\nabla u|_2 \ge \left(\frac{1}{\gamma_p C_{N,p}^p a^{p(1-\gamma_p)}}\right)^\frac{1}{p \gamma_p -2}.
\end{equation}
Now, since $\lambda_n \to \lambda$ and $u_n \to u \not \equiv 0$ weakly in $H$, and strongly in $L^p \cap L^q$, equation \eqref{768} implies that $u$ is a weak radial (and real) solution to 
\begin{equation}\label{3082}
-\Delta u = \lambda u + \mu |u|^{q-2} u + |u|^{p-2} u \qquad \text{in $\R^N$}.
\end{equation}
By the Pohozaev identity, we infer that $P_\mu(u) = 0$, i.e.
\begin{equation}\label{3083}
\gamma_p |u|_p^p =|\nabla u|_2^2 - \mu \gamma_q |u|_q^q.
\end{equation}
Testing \eqref{3082} with $u$, and using \eqref{3083}, we obtain
\begin{align*}
\lambda |u|_2^2 = \left( 1- \frac{1}{\gamma_p}\right) |\nabla u|_2^2 + \mu \left(\frac{\gamma_q}{\gamma_p} -1\right)|u|_q^q,
\end{align*}
where $1-1/\gamma_p<0$ since $\gamma_p<1$, while $\mu(\gamma_q/\gamma_p -1)>0$ since $\mu<0$ and $\gamma_q < \gamma_p$. Using again the Gagliardo-Nirenberg inequality, the fact that $|u|_2 \le a$, and estimate \eqref{3081}, we infer that
\begin{equation*}
\begin{split}
\lambda |u|_2^2 & \le \left( 1- \frac{1}{\gamma_p}\right) |\nabla u|_2^2  + \mu \left(\frac{\gamma_q}{\gamma_p} -1\right) C_{N,q}^ q |u|_2^{q(1-\gamma_q)} |\nabla u|_2^{q \gamma_q} \\
& \le |\nabla u|_2^{q \gamma_q} \left[ \left( 1- \frac{1}{\gamma_p}\right) |\nabla u|_2^{2-q \gamma_q} + \mu \left(\frac{\gamma_q}{\gamma_p} -1\right) C_{N,q}^ q a^{q(1-\gamma_q)}\right] \\
& \le |\nabla u|_2^{q \gamma_q} \left[ \left( 1- \frac{1}{\gamma_p}\right) \left(\frac{1}{\gamma_p C_{N,p}^p a^{p(1-\gamma_p)}}\right)^\frac{2-q\gamma_q}{p \gamma_p -2} +  \left(1-\frac{\gamma_q}{\gamma_p} \right) C_{N,q}^ q |\mu| a^{q(1-\gamma_q)}\right]
\end{split}
\end{equation*}
It is not difficult to check that the right hand side is strictly negative if \eqref{hp mu < 0 conv} holds, finally implying that $\lambda<0$, as desired. 

\medskip

\noindent  \textbf{Step 4) Conclusion.} By weak convergence, \eqref{768} implies that 
\begin{equation}\label{769}
dE_\mu(u)\varphi - \lambda \int_{\R^N} u \bar \varphi = 0
\end{equation}
for every $\varphi \in H$. Choosing $\varphi = u_n -u$ in \eqref{768} and \eqref{769}, and subtracting, we obtain
\[
(dE_\mu(u_n)- dE_\mu(u))[u_n-u] - \lambda \int_{\R^N} |u_n - u|^2 = o(1).
\]
Using the strong $L^p$ and $L^q$ convergence of $u_n$, we infer that 
\[
\int_{\R^N} |\nabla (u_n-u)|^2 - \lambda |u_n - u|^2 = o(1)
\]
which, being $\lambda<0$, establishes the strong convergence in $H$. 
\end{proof}

In order to deal with the dimension $N=1$, we need a variant of Lemma \ref{lem: conv PS subcr}:

\begin{lemma}\label{lem: conv PS subcr N=1}
Let $N \ge 1$, and $2 <q \le  \bar p<p<+\infty$. Let $\{u_n\} \subset S_{a}$ be a Palais-Smale sequence for $E_\mu|_{S_{a}}$ at level $c \neq 0$, and suppose in addition that: 
\begin{itemize}
\item[($i$)] $P_\mu(u_n) \to 0$ as $n \to \infty$.
\item[($ii$)] There exists $\{v_n\} \subset S_{a,r}$, with $v_n$ radially decreasing, such that $\|v_n - u_n\| \to 0$ as $n \to \infty$.
\item[($ii$)] Either $\mu>0$ (without any additional assumption), or $\mu<0$ and \eqref{hp mu < 0 conv} holds.
\end{itemize}
Then up to a subsequence $u_n \to u$ strongly in $H^1$, and $u \in S_a$ is a real-valued, radial and radially decreasing solution to \eqref{stat com} for some $\lambda<0$.
\end{lemma}

One can easily modify the proof developed in dimensions $N \ge 2$, observing that, even though $H^1_{\rad}(\R)$ does not embed compactly in $L^r(\R)$, compactness holds for bounded sequences of radially decreasing functions (see e.g. \cite[Proposition 1.7.1]{Caz}). We omit the details.


\section{Supercritical leading term with focusing subcritical perturbation}\label{sec: subcr-supcr}

In this section, for $2<q < \bar p<p<2^*$ and $a, \mu>0$ satisfying \eqref{cond sub sup} we prove Theorem \ref{thm: subcr-supcr}. Since $a$ and $\mu$ are fixed, we omit the dependence of $E_\mu$, $S_a$, $S_{a,r}$, $P_\mu$, $\cP_{a,\mu}$, $\Psi_u^\mu$, \dots on these quantities, writing simply $E$, $S$, $S_r$, $P$, $\cP$, $\Psi_u$, \dots. 

We consider the constrained functional $E|_{S}$. By the Gagliardo-Nirenberg inequality
\begin{equation}\label{E sub sup}
E (u) \ge \frac{1}{2} |\nabla u|_2^2 - \mu \frac{C_{N,q}^q}{q} a^{(1-\gamma_q) q} |\nabla u|_2^{\gamma_q q} -  \frac{C_{N,p}^p}{p} a^{(1-\gamma_p) p} |\nabla u|_2^{\gamma_p p},
\end{equation}
for every $u \in S$. Therefore, to understand the geometry of the functional $E|_{S}$ it is useful to consider the function $h : \R^+ \to \R$
\begin{equation*}
h(t):=  \frac{1}{2} t^2 - \mu \frac{C_{N,q}^q}{q} a^{(1-\gamma_q) q} t^{\gamma_q q} -  \frac{C_{N,p}^p}{p}a^{(1-\gamma_p) p} t^{\gamma_p p}.
\end{equation*}
Since $\mu>0$ and $\gamma_q q<2<\gamma_p p$, we have that $h(0^+) = 0^-$ and $h(+\infty) = -\infty$. The role of assumption \eqref{cond sub sup} is clarified by the following lemma.

\begin{lemma}\label{lem: struct h}
Under assumption \eqref{cond sub sup}, the function $h$ has a local strict minimum at negative level and a global strict maximum at positive level. Moreover, there exist $0<R_0<R_1$, both depending on $a$ and $\mu$, such that $h(R_0) = 0 = h(R_1)$ and $h(t) >0$ iff $t \in (R_0,R_1)$. 
\end{lemma}

\begin{proof}
For $t>0$, we have $h(t) > 0$ if and only if
\[
\varphi(t) > \frac{C_{N,q}^q}q  \mu a^{(1-\gamma_q) q}, \quad \text{with} \quad \varphi(t):= \frac12 t^{2-\gamma_q q} - \frac{C_{N,p}^p}p  a^{(1-\gamma_p) p} t^{\gamma_p p- \gamma_q q}.
\]
It is not difficult to check that $\varphi$ has a unique critical point on $(0,+\infty)$, which is a global maximum point at positive level, in 
\begin{equation}\label{def C_1}
\bar t:= C_1  a^{-  \frac{(1-\gamma_p)p}{\gamma_p p -2}}, \quad \text{with} \quad C_1:= \left( \frac{p(2-\gamma_q q)}{2C_{N,p}^p (\gamma_p p-\gamma_q q)}   \right)^{\frac{1}{\gamma_p p -2}}; 
\end{equation}
the maximum level is
\begin{equation*}
\varphi(\bar t) = C_2 \left(a^{  -(1-\gamma_p)p}\right)^\frac{ 2-\gamma_q q}{\gamma_p p -2}, \quad \text{with} \quad C_2:= \left( \frac{p(2-\gamma_q q)}{2C_{N,p}^p (\gamma_p p -\gamma_q q)} \right)^{\frac{2-\gamma_q q}{\gamma_p p-2}} \left( \frac{ \gamma_p p-2}{2(\gamma_p p- \gamma_q q)} \right).
\end{equation*}
Therefore, $h$ is positive on an open interval $(R_0,R_1)$ iff $\varphi(\bar t) > C_{N,q}^q \mu a^{(1-\gamma_q) q}/q$, that is \eqref{cond sub sup} holds. It follows immediately that $h$ has a global maximum at positive level in $(R_0,R_1)$. Moreover, since $h(0^+) = 0^-$, there exists a local minimum point at negative level in $(0,R_0)$. The fact that $h$ has no other critical points can be verified observing that $h'(t) = 0$ if and only if
\[
\psi(t) = \mu \gamma_q C_{N,q}^q a^{(1-\gamma_q)q}, \quad \text{with} \quad \psi(t) = t^{2-\gamma_q q} - \gamma_p C_{N,p}^p a^{(1-\gamma_p)p} t^{\gamma_p p-\gamma_q q}.
\]
Clearly $\psi$ has only one critical point, which is a strict maximum, and hence the above equation has at most two solutions, which necessarily are the local minimum and the global maximum of $h$ previously found.
\end{proof}
\begin{remark}\label{rem: st R_0}
For future convenience, we point out that in the above proof $R_0< \bar t$, with $\bar t$ defined by \eqref{def C_1}.
\end{remark}

We now study the structure of the Pohozaev manifold $\cP$. Recalling the decomposition of $\cP= \cP_+\cup \cP_- \cup \cP_0$ (see \eqref{split P}), we have:

\begin{lemma}\label{lem: struct P}
$\cP_0 = \emptyset$, and $\cP$ is a smooth manifold of codimension $2$ in $H$.
\end{lemma}
\begin{proof}
Let us assume that there exists $u \in \cP_0$. Then, combining $P(u) = 0$ with $\Psi_u''(0) = 0$ we deduce that 
\[
(2-q \gamma_q) \mu \gamma_q |u|_q^q = (p \gamma_p-2) \gamma_p |u|_p^p.
\]
Using this equation in $P(u) = 0$, we obtain both
\begin{equation}\label{761}
|\nabla u|_2^2 = \gamma_p \frac{\gamma_p p - \gamma_q q}{2-\gamma_q q} |u|_p^p \le C_{N,p}^ p\gamma_p \frac{\gamma_p p - \gamma_q q}{2-\gamma_q q} a^{(1-\gamma_p) p} |\nabla u|_2^{\gamma_p p},
\end{equation}
and
\begin{equation}\label{762}
|\nabla u|_2^2 = \mu \gamma_q \frac{\gamma_p p - \gamma_q q}{\gamma_p p-2} |u|_q^q \le \mu C_{N,q}^q \gamma_q \frac{\gamma_p p - \gamma_q q}{\gamma_p p-2} a^{(1-\gamma_q) q} |\nabla u|_2^{\gamma_q q}.
\end{equation}
From \eqref{761} and \eqref{762} we infer that
\begin{equation*}
\left( \frac{2-\gamma_q q}{C_{N,p}^ p\gamma_p (\gamma_p p - \gamma_q q)}\right)^\frac{1}{\gamma_p p -2} a^{-\frac{(1-\gamma_p) p}{\gamma_p p-2}}  \le \left( \frac{ C_{N,q}^q \gamma_q(\gamma_p p-\gamma_q q)}{ \gamma_p p - 2}\right)^\frac{1}{2-\gamma_q q} \left( \mu a^{(1-\gamma_q) q} \right)^{\frac{1}{2-\gamma_q q}},
\end{equation*}
that is
\begin{equation}\label{cond sub sup 2}
\left(\mu a^{(1-\gamma_q) q}\right)^{\gamma_p p - 2} \left(a^{(1-\gamma_p) p} \right)^{2-\gamma_q q} \ge \left( \frac{2-\gamma_q q}{C_{N,p}^p \gamma_p (\gamma_p p -\gamma_q q)} \right)^{2-\gamma_q q}\left( \frac{ \gamma_p p-2}{C_{N,q}^q \gamma_q(\gamma_p p- \gamma_q q)} \right)^{\gamma_p p-2}.
\end{equation}
It is not difficult to check that this is in contradiction with \eqref{cond sub sup}: it is sufficient to verify that the right hand side in \eqref{cond sub sup} is smaller than or equal to the right hand side in \eqref{cond sub sup 2}, and this is equivalent to
\[
\left(\frac{p \gamma_p}2\right)^{2-\gamma_q q} \left(\frac{q \gamma_	q}2\right)^{\gamma_p p-2} \le 1 
\]
for every $2<q<\bar p<p<2^*$. The validity of this estimate can be easily checked by direct computations (it is sufficient to check that $\log x/(x-1)$ is a monotone decreasing function of $x>0$). This proves that $\cP_0 = \emptyset$. 

Now we can check that $\cP$ is a smooth manifold of codimension $2$ in $H$. We note that $\cP=\{u \in H: \ P(u) = 0, \ G(u) = 0\}$, for $G(u) = |u|_2^2 -a^2$, with $P$ and $G$ of class $C^1$ in $H$. Thus, we have to show that the differential $(dG(u), dP(u)): H \to \R^2$ is surjective, for every $u \in \cP$. To this end, we prove that for every $u \in \cP$ there exists $\varphi \in T_u S$ such that $dP(u)[\varphi] \neq 0$. Once that the existence of $\varphi$ is established, the system
\[
\begin{cases}
dG(u)[\alpha \varphi + \beta u] = x \\
dP(u)[\alpha \varphi + \beta u] = y
\end{cases} \quad \iff \quad \begin{cases}
\beta a^2 = x \\
\alpha dP(u)[\varphi] + \beta dP(u)[u] = y
\end{cases}
\] 
is solvable with respect to $\alpha, \beta$, for every $(x,y) \in \R^2$, and hence the surjectivity is proved.

Now, suppose by contradiction that for $u \in \cP$ such a tangent vector $\varphi$ does not exist, i.e. $dP(u)[\varphi] = 0$ for every $\varphi \in T_u S$. Then $u$ is a constrained critical point for the functional $P$ on $S_{a}$, and hence by the Lagrange multipliers rule there exists $\nu \in \R$ such that
\[
-\Delta u = \nu u + \mu \frac{q \gamma_q}{2} |u|^{q-2} u + \frac{p \gamma_p}{2} |u|^{p-2} u \qquad \text{in $\R^N$}.
\]
But, by the Pohozaev identity, this implies that 
\[
2 |\nabla u|_2^2 = \mu q \gamma_q^2 |u|_q^q + p \gamma_p^2 |u|_p^p,
\]
that is $u \in \cP_0$, a contradiction. 
\end{proof}

The manifold $\cP$ is then divided into its two components $\cP_+$ and $\cP_-$, having disjoint closure.

\begin{lemma}\label{lem: prop fiber sub-sup}
For every $u \in S$, the function $\Psi_u$ has exactly two critical points $s_u< t_u \in \R$ and two zeros $c_u <d_u \in \R$, with $s_u<c_u<t_u<d_u$. Moreover:
\begin{itemize}
\item[1)] $s_u \star u \in \cP_+$, and $t_u \star u \in \cP_-$, and if $s \star u \in \cP$, then either $s =s_u$ or $s=t_u$. 
\item[2)] $|\nabla (s \star u)|_2 \le R_0$ for every $s \le c_u$, and
\[
E(s_u \star u) = \min \left\{E(s \star u): \ \text{$s \in \R$ and $|\nabla (s \star u)|_2 < R_0$}\right\} < 0.
\]
\item[3)] We have
\[
E(t_u \star u) = \max \left\{E(s \star u): \ s \in \R\right\} > 0,
\]
and $\Psi_u$ is strictly decreasing and concave on $(t_u,+\infty)$. In particular, if $t_u<0$, then $P(u)<0$.
\item[4)] The maps $u \in S \mapsto s_u \in \R$ and $u \in S \mapsto t_u \in \R$ are of class $C^1$.
\end{itemize}
\end{lemma}
\begin{proof}
Let $u \in S$. Then, as observed in Proposition \ref{prop: psi P}, $s \star u \in \cP$ if and only if $\Psi_u'(s) = 0$. Thus, we first show that $\Psi_u$ has at least two critical points. To this end, we recall that by \eqref{E sub sup} 
\[
\Psi_u(s) = E(s \star u) \ge h(|\nabla (s \star u)|_2) = h(e^s |\nabla u|_2). 
\]
Thus, the $C^2$ function $\Psi_u$ is positive on $(\log(R_0/|\nabla u|_2), \log(R_1/|\nabla u|_2))$, and clearly $\Psi_u(-\infty) = 0^-$, $\Psi_u(+\infty)  =-\infty$. It follows that $\Psi_u$ has at least two critical points $s_u < t_u$, with $s_u$ local minimum point on $(0, \log(R_0/|\nabla u|_2))$ at negative level, and $t_u > s_u$ global maximum point at positive level. It is not difficult to check that there are no other critical points. Indeed $\Psi_u'(s) =0$ reads
\begin{equation}\label{eq661}
\varphi(s) = \mu \gamma_q |u|_q^q, \quad \text{with} \quad \varphi(s) = |\nabla u|_2^2 e^{(2-\gamma_q q)s} -  \gamma_p |u|_p^p e^{(\gamma_p p - \gamma_q q) s}.
\end{equation}
But $\varphi$ has a unique maximum point, and hence equation \eqref{eq661} has at most two solutions. 

Collecting together the above considerations, we conclude that $\Psi_u$ has exactly two critical points: $s_u$, local minimum on $(-\infty, \log(R_0 / |\nabla u|_2))$ at negative level, and $t_u$, global maximum at positive level. By Proposition \ref{prop: psi P}, we have $s_u \star u, t_u \star u \in \cP$, and $s \star u \in \cP$ implies $s \in \{s_u,t_u\}$. By minimality $\Psi_{s_u \star u}''(0) = \Psi_u''(s_u) \ge 0$, and in fact strict inequality must hold, since $\cP_0 = \emptyset$; namely $s_u \star u \in \cP_+$. In the same way $t_u \star u \in \cP_-$.

By monotonicity and recalling the behavior at infinity, $\Psi_u$ has moreover exactly two zeros $c_u<d_u$, with $s_u<c_u<t_u<d_u$; and, being a $C^2$ function, $\Psi_u$ has at least two inflection points. Arguing as before, we can easily check that actually $\Psi_u$ has exactly two inflection points.
In particular, $\Psi_u$ is concave on $[t_u,+\infty)$, and hence, if $t_u<0$, then $P(u) = \Psi_u'(0) <0$.

It remains to show that $u \mapsto s_u$ and $u \mapsto t_u$ are of class $C^1$; to this end, we apply the implicit function theorem on the $C^1$ function $\Phi(s,u) := \Psi'_u(s)$. We use that $\Phi(s_u,u) = 0$, that $\pa_s \Phi(s_u,u) = \Psi_u''(s_u) <0$, and the fact that it is not possible to pass with continuity from $\cP_+$ to $\cP_-$ (since $\cP_0 = \emptyset$). The same argument proves that $u \mapsto t_u$ is $C^1$.
\end{proof}

For $k >0$, let us set
\[
A_k:= \left\{u \in S: |\nabla u|_2<k\right\}, \quad \text{and} \quad m(a,\mu):= \inf_{u \in A_{R_0}} E(u).
\]
As an immediate corollary, we have: 
\begin{corollary}\label{cor: norm P+}
The set $\cP_+$ is contained in $A_{R_0}=\{u \in S: |\nabla u|_2<R_0\}$, and $\sup_{\cP_+}E \le 0 \le \inf_{\cP_-}E$.
\end{corollary}

Furthermore:

\begin{lemma}\label{lem: loc inf su P}
It results that $m(a,\mu) \in (-\infty,0)$, that
\[
m(a,\mu) =  \inf_{\cP} E = \inf_{\cP_+}E, \quad \text{and that} \quad m(a,\mu) < \, \inf_{\overline{A_{R_0}} \setminus A_{R_0-\rho}} E
\]
for $\rho>0$ small enough.
\end{lemma}
\begin{proof}
For $u \in A_{R_0}$
\[
E(u) \ge h(|\nabla u|_2) \ge \min_{t \in [0,R_0]} h(t) >-\infty,
\]
and hence $m(a,\mu)>-\infty$. Moreover, for any $u \in S$ we have $|\nabla(s \star u)|_2 < R_0$ and $E(s \star u)<0$ for $s \ll-1$, and hence $m(a,\mu)<0$. 

Now, $m(a,\mu) \le \inf_{\cP_+}E$ since $\cP_+\subset A_{R_0}$ by Corollary \ref{cor: norm P+}. On the other hand, if $u \in A_{R_0}$, then $s_u \star u \in \cP_+\subset A_{R_0}$, and 
\[
E(s_u \star u) = \min\left\{E(s \star u): \ \text{$s \in \R$ and $|\nabla (s \star u)|_2 < R_0$}\right\} \le E(u),
\]
which implies that $\inf_{\cP_+}E \le m(a,\mu)$. To prove that $\inf_{\cP_+}E = \inf_{\cP} E$, it is sufficient to recall that $E>0$ on $\cP_-$, see Corollary \ref{cor: norm P+}. 

Finally, by continuity of $h$ there exists $\rho>0$ such that $h(t) \ge m(a,\mu)/2$ if $t \in [R_0-\rho, R_0]$. Therefore
\[
E(u) \ge h(|\nabla u|_2) \ge \frac{m(a,\mu)}2 > m(a,\mu)
\]
for every $u \in S$ with $R_0-\rho \le |\nabla u|_2 \le R_0$.
\end{proof}

\begin{proof}[Existence of a local minimizer]
Let us consider a minimizing sequence $\{v_n\}$ for $E|_{A_{R_0}}$. It is not restrictive to assume that $v_n \in S_r$ is radially decreasing for every $n$ (if this is not the case, we can replace $v_n$ with $|v_n|^*$, the Schwarz rearrangement of $|v_n|$, and we obtain another function in $A_{R_0}$ with $E(|v_n|^*) \le E(v_n)$). Furthermore, for every $n$ we can take $s_{v_n} \star v_n \in \cP_+$, observing that then by Lemma \ref{lem: prop fiber sub-sup} and Corollary \ref{cor: norm P+} $|\nabla (s_{v_n} \star v_n)|_2 <R_0$ and
\[
E(s_{v_n} \star v_n) = \min\left\{E(s \star v_n): \ \text{$s \in \R$ and $|\nabla (s \star v_n)|_2 < R_0$}\right\} \le E(v_n);
\]
in this way we obtain a new minimizing sequence $\{w_n=s_{v_n} \star v_n\}$, with $w_n \in S_r \cap \cP_+$ radially decreasing for every $n$. By Lemma \ref{lem: loc inf su P}, $|\nabla w_n|_2 < R_0-\rho$ for every $n$, and hence the Ekeland's variational principle yields in a standard way the existence of a new minimizing sequence $\{u_n\} \subset A_{R_0}$ for $m(a,\mu)$, with the property that $\|u_n-w_n\| \to 0$ as $n \to \infty$, which is also a Palais-Smale sequence for $E$ on $S$. The condition $\|u_n-w_n\| \to 0$ and the boundedness of $\{w_n\}$ (each $w_n$ stays in $A_{R_0}$) imply $P(u_n) \to 0$, and hence $\{u_n\}$ satisfies all the assumptions of Lemma \ref{lem: conv PS subcr N=1}: as a consequence, up to a subsequence $u_n \to \tilde u$ strongly in $H$, $\tilde u$ is an interior local minimizer for $E|_{A_{R_0}}$, and solves \eqref{stat com}-\eqref{norm} for some $\tilde \lambda<0$. The basic properties of $\tilde u$ follow directly by the convergence and by the maximum principle, and it only remains to show that $\tilde u$ is a ground state for $E|_{S}$. This follows immediately from the fact that any critical point of $E|_{S}$ lies in $\cP$, and $m(a,\mu) = \inf_{\cP} E$ (see Lemma \ref{lem: loc inf su P}).
\end{proof}


We focus now on the existence of a second critical point for $E|_{S}$.

\begin{lemma}\label{lem: 15101}
Suppose that $E(u)<m(a,\mu)$. Then the value $t_u$ defined by Lemma \ref{lem: prop fiber sub-sup} is negative.
\end{lemma}

\begin{proof}
We consider again the function $\Psi_u$, and we consider $s_u<c_u<t_u<d_u$ as in Lemma \ref{lem: prop fiber sub-sup}. If $d_u \le 0$, then $t_u<0$, and hence we can assume by contradiction that $d_u > 0$. If $0 \in (c_u,d_u)$, then $E(u) = \Psi_u(0)>0$, which is not possible since $E(u) < m(a,\mu)<0$. Therefore $c_u>0$, and by Lemma \ref{lem: prop fiber sub-sup}-(2)
\begin{align*}
m(a,\mu) & >E(u) = \Psi_u(0) \ge \inf_{s \in (-\infty,c_u]} \Psi_u(s)  \\
&\ge \inf\left\{E(s \star u): \ \text{$s \in \R$ and $|\nabla (s \star u)|_2 < R_0$}\right\} 
 =E(s_u \star u) \ge m(a,\mu),
\end{align*}
which is again a contradiction. 
\end{proof}

\begin{lemma}\label{lem: 15102}
It results that
\[
\tilde \sigma(a,\mu):=\inf_{u \in \cP_-}E(u) >0.
\]
\end{lemma}
\begin{proof}
Let $t_{\max}$ denote the strict maximum of the function $h$ at positive level, see Lemma \ref{lem: struct h}. For every $u \in \cP_-$, there exists $\tau_u \in \R$ such that $|\nabla(\tau_u \star u)|_2 = t_{\max}$. Moreover, since $u \in \cP_-$ we also have by Lemma \ref{lem: prop fiber sub-sup} that the value $0$ is the unique strict maximum of the function $\Psi_u$. Therefore
\[
E(u) =\Psi_u(0) \ge \Psi_u(\tau_u) = E(\tau_u \star u) \ge h(|\nabla (\tau_u \star u)|_2) = h(t_{\max})>0.
\]
Since $u \in \cP_-$ was arbitrarily chosen, we deduce that $\inf_{\cP_-}E \ge \max_{\R} h>0$, as desired.
\end{proof}

We shall also need the following result, where $T_u S$ denotes the tangent space to $S$ in $u$.
\begin{lemma}\label{lem: tg}
For $u\in S_a$ and $s\in\R$ the map
\[
T_{u} S \to T_{s \star u} S, \quad \varphi \mapsto s \star \varphi
\]
is a linear isomorphism with inverse $\psi\mapsto (-s) \star \psi$. \end{lemma}

For the proof, see \cite[Lemma 3.6]{BaSo2}. We can now proceed with the proof of the existence of a second positive normalized solution. In the following proof we write $E^c$ for the closed sublevel set $\{u \in S: E(u) \le c\}$.

\begin{proof}[Existence of a second critical point of mountain pass type for $E|_{S}$] We focus on the case $N \ge 2$, and we refer to Remark \ref{rmk: sub sup 1} for the necessary modification in dimension $1$.

We follow the strategy firstly introduced in \cite{Jea}, considering the augmented functional $\tilde E: \R \times H^1 \to \R$ defined by 
\begin{equation}\label{def aug}
\tilde E(s,u) := E(s \star u) = \frac{e^{2s}}{2} \int_{\R^N}|\nabla u|^2 - \mu \frac{e^{\gamma_{\bar p} \bar p s}}{\bar p} \int_{\R^N} |u|^q -  \frac{e^{2^*s}}{2^*} \int_{\R^N} |u|^{2^*},
\end{equation}
and look at the restriction $\tilde E|_{\R \times S}$. Notice that $\tilde E$ is of class $C^1$. Moreover, since $\tilde E$ is invariant under rotations applied to $u$, a Palais-Smale sequence for $\tilde E|_{\R \times S_r}$ is a Palais-Smale sequence for $E|_{\R \times S}$. 

Denoting by $E^c$ the closed sublevel set $\{u \in S: E(u) \le c\}$, we introduce the minimax class
\begin{equation}\label{def gamma}
\Gamma:= \left\{ \gamma = (\alpha,\beta) \in C([0,1], \R \times S_{r}): \ \gamma(0) \in (0,\cP_+), \ \gamma(1) \in (0,E^{2m(a,\mu)})\right\},
\end{equation}
with associated minimax level
\[
\sigma(a,\mu):= \inf_{\gamma \in \Gamma} \max_{(s,u) \in \gamma([0,1])} \tilde E(s,u).
\]
Let $u \in S_r$. By Lemma \ref{lem: prop fiber sub-sup}, there exists $s_1 \gg 1$ such that 
\begin{equation}\label{gam u}
\gamma_u: \tau \in [0,1] \mapsto (0,((1-\tau) s_u + \tau s_1) \star u) \in \R \times S_r
\end{equation}
is a path in $\Gamma$ (the continuity follows from \eqref{cont star}). Then $\sigma(a,\mu)$ is a real number.

We claim that 
\begin{equation}\label{link P-}
\text{for every $\gamma \in \Gamma$ there exists $\tau_\gamma \in (0,1)$ such that $\alpha(\tau_\gamma) \star \beta(\tau_\gamma) \in \cP_-$}.
\end{equation}
Indeed, since $\gamma(0)=(0, \beta(0)) \in (0,\cP_+)$, we have by Proposition \ref{prop: psi P} and Lemma \ref{lem: prop fiber sub-sup}
\[
t_{\alpha(0) \star \beta(0)} = t_{\beta(0)} > s_{\beta(0)} = 0.
\] 
Also, since $E(\beta(1)) = \tilde E(\gamma(1))\le 2m(a,\mu)$, we have 
\[
t_{\alpha(1) \star \beta(1)} = t_{\beta(1)} <0,
\] 
see Lemma \ref{lem: 15101}. And moreover the map $t_{\alpha(\tau) \star \beta(\tau)}$ is continuous in $\tau$, by \eqref{cont star} and Lemma \ref{lem: prop fiber sub-sup}. It follows that there exists $\tau_\gamma \in (0,1)$ such that $t_{\alpha(\tau_\gamma) \star \beta(\tau_\gamma)} =0$, that is, claim \eqref{link P-} holds. 

This implies that 
\[
\max_{\gamma([0,1])} \tilde E \ge \tilde E(\gamma(\tau_\gamma)) =E( \alpha(\tau_\gamma) \star \beta(\tau_\gamma)) \ge \inf_{\cP_-\cap S_r} E,
\]
and consequently $\sigma(a,\mu) \ge \inf_{\cP_- \cap S_r} E$. On the other hand, if $u \in \cP_- \cap S_r$, then $\gamma_u$ defined in \eqref{gam u} is a path in $\Gamma$ with 
\[
E(u) = \tilde E(0,u) = \max_{\gamma_u([0,1])} \tilde E \ge \sigma(a,\mu),
\]
whence the reverse inequality $\inf_{\cP_- \cap S_r} E \ge \sigma(a,\mu)$ follows. This, Corollary \ref{cor: norm P+} and Lemma \ref{lem: 15102} imply that
\begin{equation}\label{min max adv}
\sigma(a,\mu) = \inf_{\cP_- \cap S_r} E  >0 \ge \sup_{(\cP_+\cup E^{2m(a,\mu)}) \cap S_r}  E = \sup_{((0,\cP_+)\cup (0,E^{2m(a,\mu)})) \cap S_r}  \tilde E.
\end{equation}
Using the terminology in \cite[Section 5]{Gho}, this means that $\{\gamma([0,1]): \ \gamma \in \Gamma\}$ is a homotopy stable family of compact subsets of $\R \times S_r$ with extended closed boundary $(0,\cP_+) \cup (0,E^{0})$, and that the superlevel set $\{\tilde E \ge \sigma(a,\mu)\}$ is a dual set, in the sense that assumptions (F'1) and (F'2) in \cite[Theorem 5.2]{Gho} are satisfied. Therefore, taking any minimizing sequence $\{\gamma_n=(\alpha_n,\beta_n)\} \subset \Gamma$ for $\sigma(a,\mu)$ with the property that $\alpha_n \equiv 0$ and $\beta_n(\tau) \ge 0$ a.e. in $\R^N$ for every $\tau \in [0,1]$\footnote{Notice that, if $\{\gamma_n=(\alpha_n,\beta_n)\} \subset \Gamma$ is a minimizing sequence, then also $\{(0,\alpha_n \star |\beta_n|)\}$ has the same property.}, there exists a Palais-Smale sequence $\{(s_n,w_n\} \subset \R \times S_r$ for $\tilde E|_{\R \times S_r}$ at level $\sigma(a,\mu)$, that is 
\begin{equation}\label{1892}
\pa_s \tilde E(s_n,w_n) \to 0 \quad \text{and} \quad \|\pa_u \tilde E(s_n,w_n)\|_{(T_{w_n} S_{r})^*} \to 0 \quad \text{as $n \to \infty$},
\end{equation}
with the additional property that
\begin{equation}\label{1893}
|s_n|  + \dist_{H^1}(w_n, \beta_n([0,1])) \to 0 \qquad \text{as $n \to \infty$}.
\end{equation}
By \eqref{def aug}, the first condition in \eqref{1892} reads $P(s_n \star w_n) \to 0$, while the second condition gives
\[
e^{2s_n} \int_{\R^N} \nabla w_n \cdot \nabla \varphi - \mu e^{\gamma_q q s_n} \int_{\R^N} |w_n|^{q-2}w_n \varphi - e^{\gamma_p p s_n} \int_{\R^N} |w_n|^{p-2}w_n \varphi = o(1) \|\varphi\|
\]
for every $\varphi \in T_{w_n} S_r$, with $o(1) \to 0$ as $n \to \infty$. 
Since $\{s_n\}$ is bounded from above and from below, due to \eqref{1893}, this is equivalent to 
\begin{equation}\label{9105}
dE(s_n \star w_n)[s_n \star \varphi] = o(1) \|\varphi\| = o(1) \|s_n \star \varphi\| \quad \text{as $n \to \infty$, for every $\varphi \in T_{w_n} S_r$}.
\end{equation}
Let then $u_n:= s_n \star w_n$. By Lemma \ref{lem: tg}, equation \eqref{9105} establishes that $\{u_n\} \subset S_{r}$ is a Palais-Smale sequence for $E|_{S_{r}}$ (thus a PS sequence for $E|_S$, since the problem is invariant under rotations) at level $\sigma(a,\mu) >0$, with $P(u_n) \to 0$. By Lemma \ref{lem: conv PS subcr}, up to a subsequence $u_n \to \hat u$ strongly in $H^1$, with $\hat u \in S$ real-valued radial solution to \eqref{stat com} for some $\hat \lambda<0$. From \eqref{1893}, we have that $\hat u \ge 0$ a.e. in $\R^N$, and the strong maximum principle finally implies that $\hat u>0$ in $\R^N$. 
\end{proof}

\begin{remark}\label{rmk: sub sup 1}
In order to extend the previous proof to the $1$ dimensional case, it is natural to replace the minimizing sequence $\gamma_n = (\alpha_n, \beta_n): [0,1] \to \R \to S_r$ with $\gamma_n^* :=(0,\alpha_n \star |\beta_{n}|^*$). This is a natural candidate to be a minimizing sequence, and the second component of $\gamma_n^*$ is radially symmetric and decreasing for every $t \in [0,1]$, for every $n$. In order to check that $\gamma_n^* \in \Gamma$, we have to check that each $\gamma_n^*$ is continuous on $[0,1]$, and this issue boils down to the continuity of the symmetric decreasing rearrangement map from $H^1_+(\R^N)$ to $H^1_+(\R^N)$. Such continuity is true in $\R$, as proved in \cite{Cor}, and allows to complete the proof of the existence of $\hat u$ (using Lemma \ref{lem: conv PS subcr N=1} instead of Lemma \ref{lem: conv PS subcr}) also in dimension $N=1$. Remarkably, the symmetric decreasing rearrangement map \emph{is not continuous from $H^1_+(\R^N)$ to $H^1_+(\R^N)$ if $N \ge 2$}, see \cite{AlmLie1, AlmLie2}. This is why we treat $N=1$ and $N \ge 2$ separately.
\end{remark}


\begin{proof}[Conclusion of the proof of Theorem \ref{thm: subcr-supcr} and of Proposition \ref{prop: struct P}]
It only remains to prove that any ground state of $E|_{S}$ is a local minimizer of $E$ in $A_{R_0}$. Let then $u$ be a critical point of $E|_{S}$ with $E(u) = m(a,\mu) = \inf_{\cP} E$. Since $E(u) <0 < \inf_{\cP_-}E$, necessarily $u \in \cP_+$. Then, by Corollary \ref{cor: norm P+}, it results that $|\nabla u|_2<R_0$, and as a consequence $u$ is a local minimizer for $E$ on $A_{R_0}$.
\end{proof}

\section{Supercritical leading term with focusing critical perturbation}\label{sec: cr}

In this section we fix $N \ge 2$, $q=\bar p = 2+4/N<p<2^*$, $a, \mu >0$ satisfying \eqref{hp cr pos}, and prove Theorem \ref{thm: supcr2}. The $1$ dimensional case can be treated using the strategy described in Remark \ref{rmk: sub sup 1}. Since $a$ and $\mu$ will always be fixed, we omit the dependence on these quantities.

The change of the geometry of $E|_{S}$ with respect to the case $q<\bar p$ is enlightened by the following simple lemmas. We recall the decomposition $\cP= \cP_+ \cup \cP_0 \cup \cP_-$, see \eqref{split P}.
%

%

\begin{lemma}\label{lem: struct P cr}
We have $\cP_0 = \emptyset$, and $\cP$ is a smooth manifold of codimension $2$ in $H$.
\end{lemma}
\begin{proof}
If $u \in \cP_0$, that is $\Psi_u'(0) = \Psi_u''(0) = 0$, then necessarily $|u|_p =0$, which is not possible since $u \in S$. The rest of the proof is very similar (actually simpler) to the one of Lemma \ref{lem: struct P}, and hence is omitted.
\end{proof}

\begin{lemma}\label{lem: fiber cr}
For every $u \in S$, there exists a unique $t_u \in \R$ such that $t_u \star u \in \cP$. $t_u$ is the unique critical point of the function $\Psi_u$, and is a strict maximum point at positive level. Moreover:
\begin{itemize}
\item[1)] $\cP= \cP_-$. 
\item[2)] $\Psi_u$ is strictly decreasing and concave on $(t_u, +\infty)$, and $t_u<0$ implies $P(u)<0$.
\item[3)] The map $u \in S \mapsto t_u \in \R$ is of class $C^1$.
\item[4)] If $P(u)<0$, then $t_u<0$.
\end{itemize}
\end{lemma}

\begin{proof}
Since $q=\bar p$ and $\gamma_{\bar p} \bar p = 2$, we have that
\[
\Psi_u(s) = \left( \frac12 |\nabla u|_2^2 - \frac{\mu}{\bar p} |u|_{\bar p}^{\bar p} \right) e^{2s} - \frac1p |u|_p^p e^{\gamma_p p s};
\]
then, by Proposition \ref{prop: psi P}, to prove existence and uniqueness of $t_u$, together with monotonicity and convexity of $\Psi_u$, we have only to show that the term inside the brackets is positive. This is clearly satisfied, since
\[
\frac12 |\nabla u|_2^2 - \frac{\mu}{\bar p} |u|_{\bar p}^{\bar p} \ge \left( \frac12 - \frac{\mu}{\bar p} C_{N, \bar p}^{\bar p} a^{ 4/N}\right) |\nabla u|_2^2 >0
\]
by the Gagliardo-Nirenberg inequality and assumption \eqref{hp cr pos}. 

Now, if $u \in \cP$, then $t_u = 0$, and being a maximum point we have $\Psi_u''(0) \le 0$. In fact, since $\cP_0= \emptyset$, necessarily $\Psi_u''(0) <0$, so that $\cP= \cP_-$.

For the smoothness of $u \mapsto t_u$ we can apply the implicit function theorem as in Lemma \ref{lem: prop fiber sub-sup}.

Finally, since $\Psi_u'(t)<0$ if and only if $t>t_u$, we have that $P(u) = \Psi_u'(0)<0$ if and only if $t_u<0$.
\end{proof}

\begin{lemma}\label{lem: E su P cr}
It results that
\[
m(a,\mu):= \inf_{u \in \cP} E(u) >0.
\]
\end{lemma}

\begin{proof}
If $u \in \cP$, then $P(u) = 0$, and by the Gagliardo-Nirenberg inequality
\[
|\nabla u|_2^2 \le \gamma_p C_{N,p}^ p a^{(1-\gamma_p)p} |\nabla u|_2^{\gamma_p p} + \mu \frac{2}{\bar p} C_{N,\bar{p}}^{\bar p} a^{4/N} |\nabla u|_2^2. 
\]
As a consequence
\begin{equation}\label{lower b on P}
|\nabla u|_2^{\gamma_p p} \ge \frac{a^{-(1-\gamma_p)p}}{\gamma_p C_{N,p}^p}\left(1- \frac{2}{\bar p} C_{N,\bar{p}}^{\bar p} \mu a^{4/N}\right) |\nabla u|_2^2 \quad \implies \quad \inf_{\cP} |\nabla u|_2 >0,
\end{equation}
where we used assumption \eqref{hp cr pos}. Now, for any $u \in \cP$
\[
E(u) = \frac12\left(1-\frac2{\gamma_p p}\right) |\nabla u|_2^2 - \frac{\mu}{\bar p} \left( 1- \frac{2}{\gamma_p p}\right) |u|_{\bar p}^{\bar p} \ge \frac12\left(1-\frac2{\gamma_p p}\right)\left(1- \frac2{\bar p} C_{N,\bar p}^{\bar p} \mu a^{\frac4N}\right)|\nabla u|_2^2,
\]
and hence the thesis follows from \eqref{hp cr pos} and \eqref{lower b on P}.
\end{proof}

\begin{lemma}\label{lem: stima sup A cr}
There exists $k>0$ sufficiently small such that
\[
0<  \sup_{\overline{A_k}} E < m(a,\mu) \quad \text{and} \quad u \in \overline{A_k} \implies E(u), P(u) >0,
\]
where $A_k= \left\{u \in S: |\nabla u|_2^2 < k\right\}$. 
\end{lemma}

\begin{proof}
By the Gagliardo-Nirenberg inequality and assumption \eqref{hp cr pos}
\[
\begin{split}
E(u) &\ge \left(\frac12- \frac{1}{\bar p} C_{N,\bar p}^{\bar p} \mu a^\frac4N\right) |\nabla u|_2^2 - \frac{C_{N,p}^p}{p} a^{(1-\gamma_p)p}|\nabla u|_2^{\gamma_p p}>0, \\
P(u) &\ge \left(1- \frac{2}{\bar p} C_{N,\bar p}^{\bar p} \mu a^\frac4N\right) |\nabla u|_2^2 - C_{N,p}^p a^{(1-\gamma_p)p}|\nabla u|_2^{\gamma_p p} > 0,
\end{split}
\]
if $u \in \overline{A_k}$ with $k$ small enough. If necessary replacing $k$ with a smaller quantity, recalling that $m(a,\mu) >0$ by Lemma \ref{lem: E su P cr} we also have
\[
E(u) \le \frac12 |\nabla u|_2^2 < m(a,\mu). \qedhere
\]
\end{proof}

In what follows we prove the existence of ground state of mountain pass type at level $m_r(a,\mu):= \inf_{\cP \cap S_r} E$.

\begin{proof}[Existence of a critical point of mountain pass-type]
Let $k>0$ be defined by Lemma \ref{lem: E su P cr}. As in the previous section, we consider the augmented functional $\tilde E: \R \times H^1 \to \R$ defined by \eqref{def aug}, and the minimax class
\begin{equation}\label{def gamma}
\Gamma:= \left\{ \gamma = (\alpha,\beta) \in C([0,1], \R \times S_{r}): \ \gamma(0) \in (0,\overline{A_k}), \ \gamma(1) \in (0,E^{0})\right\},
\end{equation}
with associated minimax level
\[
\sigma(a,\mu):= \inf_{\gamma \in \Gamma} \max_{(s,u) \in \gamma([0,1])} \tilde E(s,u).
\]
Let $u \in S_r$. Since $|\nabla (s \star u)|_2 \to 0^+$ as $ s \to -\infty$, and $\Psi_u(s) \to -\infty$ as $s \to +\infty$, there exist $s_0 \ll -1$ and $s_1 \gg 1$ such that 
\begin{equation}\label{gam u cr}
\gamma_u: \tau \in [0,1] \mapsto (0,((1-\tau) s_0 + \tau s_1) \star u) \in \R \times S_r
\end{equation}
is a path in $\Gamma$ (the continuity follows from \eqref{cont star}). Then $\sigma(a,\mu)$ is a real number.

Now, for any $\gamma =(\alpha,\beta) \in \Gamma$, let us consider the function
\[
P_\gamma:\tau \in [0,1] \mapsto P(\alpha(\tau) \star \beta(\tau)) \in \R.
\]
We have $P_\gamma(0) = P(\beta(0)) >0$, by Lemma \ref{lem: stima sup A cr}, and we claim that $P_\gamma(1) =P(\beta(1))<0$: indeed, since $\Psi_{\beta(1)}(s)>0$ for every $s \in (-\infty,t_{\beta(1)}]$, and $\Psi_{\beta(1)}(0) =E(\beta(1)) \le 0$, it is necessary that $t_{\beta(1)}<0$. By Lemma \ref{lem: fiber cr}, this implies the claim. Moreover, $P_\gamma$ is continuous by \eqref{cont star}, and hence we deduce that there exists $\tau_\gamma \in (0,1)$ such that $P_\gamma(\tau_\gamma) = 0$, namely $\alpha(\tau_\gamma) \star \beta(\tau_\gamma) \in \cP$; this implies that 
\[
\max_{\gamma([0,1])} \tilde E \ge \tilde E(\gamma(\tau_\gamma)) =E( \alpha(\tau_\gamma) \star \beta(\tau_\gamma)) \ge \inf_{\cP\cap S_r} E = m_r(a,\mu),
\]
and consequently $\sigma(a,\mu) \ge m_r(a,\mu)$. On the other hand, if $u \in \cP_- \cap S_r$, then $\gamma_u$ defined in \eqref{gam u cr} is a path in $\Gamma$ with 
\[
E(u) = \max_{\gamma_u([0,1])} \tilde E \ge \sigma(a,\mu),
\]
whence the reverse inequality $m_r(a,\mu) \ge \sigma(a,\mu)$ follows. Combining this with Lemmas \ref{lem: E su P cr}, we infer that
\begin{equation}\label{stima infmax}
\sigma(a,\mu) = m_r(a,\mu) > \sup_{(\overline{A_k}  \cup E^{0}) \cap S_r} E = \sup_{((0,\overline{A_k})  \cup (0,E^{0})) \cap (\R \times S_r)} \tilde E.
\end{equation}
Using the terminology in \cite[Section 5]{Gho}, this means that $\{\gamma([0,1]): \ \gamma \in \Gamma\}$ is a homotopy stable family of compact subsets of $\R \times S_r$ with extended closed boundary $(0,\overline{A_k}) \cup (0,E^{0})$, and that the superlevel set $\{\tilde E \ge \sigma(a,\mu)\}$ is a dual set for $\Gamma$, in the sense that assumptions (F'1) and (F'2) in \cite[Theorem 5.2]{Gho} are satisfied. The existence of a positive real valued $\tilde u \in S_a$ solving \eqref{stat com} follows now as in the proof of Theorem \ref{thm: subcr-supcr} (existence of $\hat u$).
\end{proof}

\begin{proof}[Conclusion of the proof of Theorem \ref{thm: supcr2} and Proposition \ref{prop: struct P}]
To verify that $\tilde u$ is a ground state, we show that $\tilde u$ achieves $\inf_{\cP} E =m(a,\mu)$. From our proof, we know that $\sigma(a,\mu) = E(\tilde u) = \inf_{\cP \cap S_r} E \ge m(a,\mu)$, and hence we have to show that also the reverse inequality holds. This amounts to verify that $\inf_{\cP \cap S_r} E \le \inf_{\cP} E$. Suppose by contradiction that there exists $u \in \cP \setminus S_r$ with $E(u) < \inf_{\cP \cap S_r} E$. Then we let $v:= |u|^*$, the symmetric decreasing rearrangement of the modulus of $u$, which lies in $S_r$. By standard properties $|\nabla v|_2 \le |\nabla u|_2$, $E(v) \le E(u)$, and $P(v) \le P(u) = 0$. If $P(v) = 0$ we immediately have a contradiction, and hence we can assume that $P(v) <0$. In this case, from Lemma \ref{lem: fiber cr} we know that $t_v<0$. But then we obtain a contradiction in the following way:
\begin{align*}
E(u) &< E(t_v \star v) = e^{2 t_v} \left( \frac12\left(1- \frac{2}{\gamma_p p}\right) |\nabla v|_2^2 - \frac{\mu}{\bar p}\left(1- \frac{2}{\gamma_p p}\right) |v|_{\bar p}^{\bar p}\right) \\
& \le e^{2 t_v} \left( \frac12\left(1- \frac{2}{\gamma_p p}\right) |\nabla u|_2^2 - \frac{\mu}{\bar p}\left(1- \frac{2}{\gamma_p p}\right) |u|_{\bar p}^{\bar p}\right)  = e^{2 t_v} E(u) < E(u),
\end{align*}
where we used the fact that $s_v \star v$ and $u$ lies in $\cP$. This proves that $\inf_{\cP \cap S_r} E=\inf_{\cP} E$, and hence $\tilde u$ is a ground state. 
\end{proof}

\section{Supercritical leading term with defocusing perturbation}\label{sec: mu<0}

In this section we prove Theorem \ref{thm: sup mu <0} for $N \ge 2$. Since $a$ and $\mu$ are fixed, we omit again the dependence on these quantities. We consider once again the Pohozaev manifold $\cP$, defined in \eqref{def P}, and the decomposition $\cP= \cP_+\cup \cP_0 \cup \cP_-$, see \eqref{split P}.
%

%

\begin{lemma}\label{lem: struct P mu <0}
We have $\cP_0 = \emptyset$, and $\cP$ is a smooth manifold of codimension $2$ in $H$.
\end{lemma}
\begin{proof}
If $u \in \cP_0$, then
\[
\mu \gamma_q(2- \gamma_q q)|u|_q^q = \gamma_p (p \gamma_p -2)|u|_p^p,
\]
which implies $u \equiv 0$ since $\mu<0$ and $\gamma_q q \le 2 < \gamma_p p$. This contradicts the fact that $u \in S_a$. The rest of the proof is very similar to the one of Lemma \ref{lem: struct P}, and hence is omitted.
\end{proof}

\begin{lemma}\label{lem: fiber mu < 0}
For every $u \in S$, there exists a unique $t_u \in \R$ such that $t_u \star u \in \cP$. $t_u$ is the unique critical point of the function $\Psi_u$, and is a strict maximum point at positive level. Moreover:
\begin{itemize}
\item[1)] $\cP= \cP_-$. 
\item[2)] $\Psi_u$ is strictly decreasing and concave on $(t_u, +\infty)$, and $t_u<0$ implies $P(u)<0$.
\item[3)] The map $u \in S \mapsto t_u \in \R$ is of class $C^1$.
\item[4)] If $P(u)<0$, then $t_u<0$.
\end{itemize}
\end{lemma}

%
%
%

\begin{proof}
Notice that, since $\mu<0$, we have $\Psi_u(s) \to 0^+$ as $s \to -\infty$, and $\Psi_u(s) \to -\infty$ as $s \to +\infty$, for every $u \in S_r$. Therefore, $\Psi_u$ has a global maximum point at positive level. To show that this is the unique critical point of $\Psi_u$, we observe that $\Psi_u'(s) = 0$ if and only if
\[
|\nabla u|_2^2 e^{(2-\gamma_q q) s}    - \gamma_p |u|_p^p e^{(\gamma_p p-2)s} = -|\mu| \gamma_q |u|_q^q < 0,
\]
and, since the right hand side is negative, this equation has only one solution. In the same way, one can also check that $\Psi_u$ has only one inflection point. Since $\Psi_u'(t)<0$ if and only if $t>t_u$, we have that $P(u) = \Psi_u'(0)<0$ if and only if $t_u<0$. Finally, for point (3) we argue as in Lemma \ref{lem: prop fiber sub-sup}.
\end{proof}

\begin{lemma}
It results that 
\[
m(a,\mu):= \inf_{u \in \cP} E(u) >0.
\]
\end{lemma}

\begin{proof}
If $u \in \cP$, then by \eqref{Poh}
\[
|\nabla u|_2^2 \le \gamma_p |u|_p^p \le \gamma_p C_{N,p}^p a^{(1-\gamma_p) p} |\nabla u|_2^{\gamma_p p},
\]
whence we deduce that $\inf_{\cP} |\nabla u|_2 \ge C_1 >0$. At this point it is sufficient to observe that, always by \eqref{Poh} 
\[
E(u) = \frac12\left(1-\frac2{\gamma_p p}\right) |\nabla u|_2^2 - \frac{\mu}q \left( 1- \frac{\gamma_q q}{\gamma_p p}\right) |u|_q^q \ge \frac12\left(1-\frac2{\gamma_p p}\right) |\nabla u|_2^2,
\]
and the thesis follows.
\end{proof}

\begin{lemma}\label{lem: stima sup A mu < 0}
There exists $k>0$ sufficiently small such that
\[
0<  \sup_{\overline{A_k}} E < m(a,\mu) \quad \text{and} \quad u \in \overline{A_k} \implies E(u), P(u) >0,
\]
where $A_k:= \left\{u \in S: |\nabla u|_2^2 < k\right\}$.
\end{lemma}


\begin{proof}
By the Gagliardo-Nirenberg inequality
\[
\begin{split}
E(u) & \ge \frac12 |\nabla u|_2^2 - \frac{C_{N,p}^p}{p} a^{(1-\gamma_p)p}|\nabla u|_2^{\gamma_p p} >0, \\
P(u) &\ge  |\nabla u|_2^2 - \frac{C_{N,p}^p}{p} a^{(1-\gamma_p)p}|\nabla u|_2^{\gamma_p p} > 0,
\end{split}
\]
if $u \in \overline{A_k}$ with $k$ small enough. If necessary replacing $k$ with a smaller quantity, we also have
\[
E(u) \le \frac12 |\nabla u|_2^2 + C |\mu| a^{(1-\gamma_q)q} |\nabla u|_2^{\gamma_q q} <m(a,\mu). \qedhere
\]
\end{proof}

%
%

\begin{proof}[Proof of Theorem \ref{thm: sup mu <0} and Proposition \ref{prop: struct P}]
We can proceed exactly as in the proof of Theorem \ref{thm: supcr2}, using Lemmas \ref{lem: fiber mu < 0} and \ref{lem: stima sup A mu < 0} instead of Lemmas \ref{lem: fiber cr} and \ref{lem: stima sup A cr}, respectively. In this way we prove the existence of a ground state $\tilde u$ for $E$ on $S_a$ at the positive level $\inf_{\cP} E$. We omit the details.
\end{proof}

\begin{remark}\label{rem: mu = 0}
The proofs in this section clearly work also for the homogeneous problem $\mu=0$. In particular, we recover the following facts which are essentially known (see \cite{Caz,Jea}):
\begin{itemize}
\item For any $u \in S_a$, there exists a unique $t_{u,0} \in \R$ such that $t_{u,0} \star u \in \cP_{a,0}$. $t_{u,0}$ is the unique critical point of the function $\Psi^0_u$, and is a strict maximum point at positive level. In particular, $\cP_{a,0}= \cP_-^{a,0}$.
\item For every $a>0$, there exists a ground state of $E_0$ on $S_a$ at a positive level $m(a,0) = \inf_{\cP} E_0 = \inf_{\cP_-} E_0$.
\end{itemize}
\end{remark}


\section{Properties of ground states I}\label{sec: prop I}

In this section we focus on the properties of ground states in the supercritical-subcritical setting with $\mu>0$. In Subsection \ref{sec: desc Z}, we prove the stability and the characterization of $Z_{a,\mu}$ in Theorem \ref{thm: Z supcr 1}. The strong instability of the standing wave $e^{-i \hat \lambda t} \hat u(x)$ is the content of Subsection \ref{sec: inst 1}. The asymptotic behavior of ground states, Theorems \ref{thm: mu to 0} and \ref{thm: q to bar p}, is addressed in Subsections \ref{sub: as 1} and \ref{sub: as 2}.

\subsection{Description of $Z_{a,\mu}$}\label{sec: desc Z}

Let now $a>0$ be fixed, and let $\mu>0$ satisfy \eqref{cond sub sup}. In order to prove the orbital stability of $Z_{a,\mu}$, a crucial intermediate step is the relative compactness of all the minimizing sequences for $m(a,\mu) = \inf_{A_{R_0}} E_\mu$, up to translations. In general minimizing sequences will not have the special properties of Lemma \ref{lem: conv PS subcr}. However, this obstruction can be overcome using a very nice idea of M. Shibata \cite{Shi}. 

As a preliminary observation, we note that for the ground state level $m(a,\mu)$, which was characterized as $\inf_{A_{R_0}} E_\mu$, the stronger characterization 
\begin{equation}\label{lem 8.2}
m(a,\mu) =\inf_{A_{R_1}} E = \inf \{E(u): \ u \in S,\ |\nabla u|_2 < R_1\}
\end{equation}
holds. Indeed, if $|\nabla u|_2 \in [R_0, R_1]$, then $E_\mu(u) \ge h(|\nabla u|_2) > 0 > \inf_{A_{R_0}} E_\mu$, see \eqref{E sub sup} and Lemma \ref{lem: struct h}. Notice that the values $R_0$ and $R_1$ depend on $a$ and $\mu$ by means of Lemma \ref{lem: struct h}. In this subsection we stress this dependence writing $R_0(a,\mu)$ and $R_1(a,\mu)$. Analogously, the definition of $A_{R_0}$ depends on $a$ and on $\mu$, and hence we explicitly write $A_{a,R_0(a,\mu)}$ in what follows.

\begin{lemma}\label{lem: subadd 1}
Let $\tilde a, \rho >0$. There exists $\tilde \mu=\tilde \mu(\tilde a+ \rho)>0$ such that, if $0<a\le \tilde a$ and $0< \mu < \tilde \mu$, then:
\begin{itemize}
\item[$i$)] $2R_0^2(\tilde a + \rho, \mu) < R_1^2(\tilde a,\mu)$.
\item[$ii$)] for any $a_1,a_2>0$ with $a_1^2 + a_2^2 = a^2$, we have
\[
R_0^2(a_1,\mu) + R_0^2(a_2,\mu) < R_1^2(a,\mu).
\]
\item[$iii$)] The functions $(a,\mu) \mapsto R_0(a,\mu)$ and $(a ,\mu)\mapsto R_1(a,\mu)$ are of class $C^1$ in $(0,\tilde a+\rho) \times (0,\tilde \mu)$, $R_0$ is monotone increasing in $a$, while $R_1(a,\mu)$ is monotone decreasing in $a$.
\end{itemize}
\end{lemma}
\begin{proof}
We recall that, by Lemma \ref{lem: struct h}, $0<R_0=R_0(a,\mu)< R_1=R_1(a,\mu)$ are the roots of $g(t, a,\mu) = 0$, with 
\[
g(t, a,\mu) := \frac12 t^{2-\gamma_q q} - \frac{C_{N,p}^p}p a^{(1-\gamma_p)p} t^{\gamma_p p- \gamma_q q} - \frac{C_{N,q}^q}{q} \mu a^{(1-\gamma_q)q} = \varphi (t,a)-  \frac{C_{N,q}^q}{q} \mu a^{(1-\gamma_q)q},
\]
where we recall the definition of $\varphi = \varphi(\cdot\,,a)$ from Lemma \ref{lem: struct h}; the existence of $R_0$ and $R_1$ is guaranteed by assumption \eqref{cond sub sup}. Let then $\tilde a, \rho>0$, and consider the range of $\mu>0$ such that \eqref{cond sub sup} is satisfied with $a= \tilde a + \rho$. This range contains a right neighborhood of $0$. Taking the limit as $\mu \to 0^+$, by continuity we have that $R_0(\tilde a + \rho,\mu)$ and $R_1(\tilde a+ \rho,\mu)$ converge, respectively, to $0$ and to the only positive root of $\varphi (t,\tilde a +\rho) = 0$. In particular, for every $\tilde a, \rho>0$ fixed there exists $\tilde \mu=\tilde \mu(\tilde a+\rho)>0$ such that 
\begin{equation}\label{10102}
2R_0(\tilde a+\rho,\mu)^2 <   R_1^2(\tilde a+\rho,\mu) \quad \text{whenever $0< \mu<\tilde \mu$}.
\end{equation}

Let now $0<a \le \tilde a+\rho$ and $0<\mu<\tilde \mu$. Under assumption \eqref{cond sub sup}, we have that 
\[
\pa_t g(t,a, \mu) = \varphi'(t, a).
\]
We checked that $\varphi'(\cdot\,,a)$ has a unique critical point on $(0,+\infty)$, which is a strict maximum point, in $\bar t= \bar t(a)$, with $0<R_0< \bar t< R_1$, and hence in particular $\pa_t g(R_0(a,\mu), a,\mu) > 0$. Thus, the implicit function theorem implies that $R_0(a,\mu)$ is a locally unique $C^1$ function of $(a,\mu)$, with
\[
\frac{\pa R_0(a,\mu)}{\pa a} = - \frac{\pa_a g(R_0(a,\mu),a,\mu)}{\pa_t g(R_0(a,\mu),a,\mu)} >0.
\]
In a similar way, one can show that $R_1(a,\mu)$ is a locally unique $C^1$ function of $(a,\mu)$, with $\pa_a R_1(a,\mu) < 0$. In particular, $R_0$ is monotone increasing and $R_1$ is monotone decreasing in $a$, and using the monotonicity of $R_1$ in \eqref{10102}, point ($i$) of the lemma follows. Concerning point ($ii$), if $a_1^2 + a_2^2 = a^2 < \tilde a^2$, we deduce that
\[
R_0^2(a_1,\mu) + R_0^2(a_2,\mu) < 2R_0^2(a,\mu) <2 R_0^2(\tilde a+\rho,\mu)< R_1^2(\tilde a,\mu) < R_1^2(a,\mu),
\]
as desired. 
%
%
%
\end{proof}

\begin{remark}\label{rem: on tilde mu}
For any $a>0$, the positive value of $\tilde \mu$ appearing in Theorem \ref{thm: Z supcr 1} is the maximum $\tilde \mu >0$ such that \eqref{10102} holds. We believe that from this condition it is possible to obtain more explicit estimates, but we decided to not insist on this point.
\end{remark}

Using the coupled rearrangement introduced by M. Shibata \cite[Section 2.2]{Shi}, it is now possible to prove strict subadditivity for $m(a,\mu)$. 

In what follows, for a fixed $a>0$, we take an arbitrarily small $\rho>0$ and consider $\tilde \mu= \tilde \mu(a+\rho)$, defined in Lemma \ref{lem: subadd 1}

\begin{lemma}\label{lem: subadd 2}
If $a_1^2 + a_2^2 =a^2$ and $0<\mu<\tilde \mu$, then 
\[
m(a,\mu) < m(a_1,\mu) + m(a_2,\mu).
\]
\end{lemma} 
\begin{proof}
Let $v$ and $w$ two real-valued, positive, radially symmetric and radially decreasing minimizers for $m(a_1,\mu)$ and $m(a_2,\mu)$, obtained by Theorem \ref{thm: subcr-supcr}. Then $v$ and $w$ are solutions to \eqref{stat com} for some $\lambda_1, \lambda_2<0$, and in particular are of class $C^2$ (by regularity) and are strictly positive in $\R^N$ (by the maximum principle). Therefore, by \cite[Lemma 2.2-Theorem 2.4]{Shi} there exists a function $u \in H^1$ (the coupled rearrangement of $v$ and $w$) such that
\[
|u|_r^r = |v|_r^r + |w|_r^r \quad \forall r \ge 1, \quad \text{and} \quad |\nabla u|_2^2 < |\nabla v|_2^2 + |\nabla w|_2^2.
\]
Notice that we have strict inequality for the norm of the gradients. As a consequence, we have that $u \in S_a$, 
\[
|\nabla u|_2^2 < R_0^2(a_1,\mu) + R_0^2(a_2,\mu) < R_1^2(a,\mu)
\]
by Lemma \ref{lem: subadd 1}, and hence recalling \eqref{lem 8.2}
\[
m(a,\mu) = \inf_{A_{a,R_1(a,\mu)}}E_\mu \le E_\mu(u) < E_\mu(v) + E_\mu(w) = m(a_1,\mu) + m(a_2,\mu),
\]
as desired.
\end{proof}

\begin{proposition}\label{prop: rel cpt sub 2}
In the previous setting, any sequence $\{u_n\} \subset H^1$ such that
\[
E_\mu(u_n) \to m(a,\mu), \quad |u_n|_2 \to a, \quad |\nabla u_n|_2 < R_0(a + \rho,\mu) 
\]
is relatively compact in $H^1$ up to translations.
\end{proposition}

\begin{proof}
As in the proof of Proposition \ref{prop: rel cpt sub}, by concentration-compactness we have three alternatives: either vanishing, or dichotomy, or else compactness holds for the scaled sequence $v_n = a u_n/|u_n|_2$. 

The occurrence of vanishing can be easily ruled out, observing that if vanishing holds, then $v_n \to 0$ in $L^r$ for $r \in (2,2^*)$, and hence we would obtain $\liminf_n E_\mu(u_n) \ge 0$, in contradiction with the fact that $m(a,\mu)<0$.

We show now that also dichotomy cannot hold. Otherwise, as in Proposition \ref{prop: rel cpt sub}, we deduce that 
\begin{equation}\label{cl1011}
m(a,\mu) =\lim_{n \to \infty} E_\mu(u_n) = \lim_{n \to \infty} E_\mu(v_n) \ge \limsup_{n \to \infty} \left(E_\mu(v_n^1) + E_\mu(v_n^2) \right).
\end{equation}
We claim that
\begin{equation}\label{cl1010}
|\nabla v_n^1|_2 \le R_1(a_1,\mu) \quad \text{and} \quad  |\nabla v_n^2|_2 \le R_1(a_2,\mu).
\end{equation}
Once that the claim is proved, estimate \eqref{cl1011} gives a contradiction with the strict subadditivity in Lemma \ref{lem: subadd 2} and \eqref{lem 8.2}, and rules out the occurrence of dichotomy. To prove claim \eqref{cl1010}, we observe at first that by concentration-compactness $|\nabla v_n^1|_2^2 + |\nabla v_n^2|_2^2 \le R_0^2(a +\rho,\mu)$. Therefore, if (up to a subsequence) $|\nabla v_n^1|_2^2 > R_1(a_1,\mu)^2$, by Lemma \ref{lem: subadd 1}
\[
R_1^2(a,\mu) < R_1^2(a_1,\mu)^2 < |\nabla v_n^1|_2^2 < |\nabla v_n^1|_2^2 + |\nabla v_n^2|_2^2 \le R_0(a + \rho,\mu)^2<R_1^2(a,\mu),
\]
a contradiction. Thus, claim \eqref{cl1010} holds, and we have compactness up to translations as in the proof of Proposition \ref{prop: rel cpt sub}.
\end{proof}

\begin{proof}[Stability of ground states]
Similarly as in Lemma \ref{lem: cont gsl} (and using the continuity and monotonicity of $R_0(a,\mu)$ with respect to $a$), it is not difficult to check $m(a,\mu)$ is a continuous function of $a$. Thus, arguing as in \cite[Theorem 3.1]{HajStu}, it is possible to use Proposition \ref{prop: rel cpt sub 2} to show that any sequence $\{u_n\} \subset A_{a,R_{0}(a + \rho,\mu)}$ (not necessarily of real-valued functions) such that
\[
E_\mu(u_n) \to m(a,\mu), \quad \text{and} \quad |u_n|_2 \to a
\]
is relatively compact in $H$ up to translations.

We can now complete the proof of the stability of $Z_{a,\mu}$. Recall that we fixed $a>0$, and for any small $\rho$ we considered $\tilde \mu= \tilde \mu(a+\rho)$ and $0< \mu <\tilde \mu$. Suppose that there exists $\eps>0$, a sequence of initial data $\{\psi_{n,0}\} \subset H$ and a sequence $\{t_n\} \subset (0,+\infty)$ such that the maximal solution $\psi_n$ with $\psi_n(0, \cdot) = \psi_{n,0}$ satisfies
\begin{equation}\label{10101}
\lim_{n \to \infty} \, \inf_{v \in Z_{a,\mu}} \|\psi_{n,0}- v\| = 0 \quad \text{and} \quad \inf_{v \in Z_{a,\mu}} \|\psi_n(t_n) - v\| \ge \eps
\end{equation}
(we refer to \cite[Section 3]{TaoVisZha} for the local well-posedness for the \eqref{com nls}). Clearly $|\psi_{n,0}|_2 =: a_n \to a$ and $E_\mu(\psi_{n,0}) \to m(a,\mu)$, by continuity. Furthermore, always by continuity and using point ($i$) of Lemma \ref{lem: subadd 1}, we deduce that $|\nabla \psi_{n,0}|_2 < R_0(a+\rho,\mu)< R_1(a_n,\mu)$ for every $n$ sufficiently large. Since $|\nabla \psi_{n,0}|_2 \in [R_0(a_n,\mu), R_1(a_n,\mu)]$ implies that $E_\mu(\psi_{n,0})  \ge 0$, we deduce that in fact $|\nabla \psi_{n,0}|_2 < R_0 (a_n,\mu)<R_0(a+\rho,\mu)$.

Let us consider now the solution $\psi_n(t, \cdot)$. Since $\psi_{n,0} \in A_{a_n, R_0(a_n,\mu)}$, if $\psi_n(t, \cdot)$ exits from $A_{a_n,R_0(a_n,\mu)}$ there exists $t \in (0,T_{\max})$ such that $|\nabla \psi_n(t,\cdot)|_2 = R_0(a_n,\mu)$; but then $E_\mu(\psi_n(t, \cdot)) \ge h(R_0) = 0$, against the conservation of energy. This shows that solutions starting in $A_{a_n, R_0(a_n,\mu)}$ are globally defined in time and satisfy $|\nabla \psi_{n}(t,\cdot)|_2 <R_0(a_n,\mu)<R_0(a+\rho,\mu)$ for every $t \in (0,+\infty)$. Moreover, by conservation of mass and of energy $|\psi_n(t, \cdot)|_2 \to a$, and $E_\mu(\psi_n(t_n, \cdot)) \to m(a,\mu)$ as $n \to \infty$. It follows that $\{\psi_{n}(t_n,\cdot)\}$ is relatively compact up to translations in $H$, and hence it converges, up to a translation, to a ground state in $Z_{a,\mu}$, in contradiction with \eqref{10101}.
\end{proof}

\begin{proof}[Structure of $Z_{a,\mu}$]
Let $u \in Z_{a,\mu}$ be a ground state for $E_\mu|_{S_a}$: $|\nabla u|_2 < R_0(a,\mu)$ and $E_\mu(u) = m(a,\mu)$. Then $|u|$ satisfies $|\nabla |u||_2 \le |\nabla u|_2 < R_0(a,\mu)$ and $E_\mu(|u|) \le E_\mu(u) = m(a,\mu)$. It follows that $|u|$ is a non-negative real-valued ground state as well, with $|\nabla |u||_2 = |\nabla u|_2$; in particular, it satisfies \eqref{stat com} and hence it is of class $C^2$ and is positive in $\R^N$. At this point it possible to argue as in \cite[Section 4]{HajStu}, completing the proof.
\end{proof}

This completes the proof of the first part of Theorem \ref{thm: Z supcr 1}.

\subsection{Strong instability of the $e^{-i \hat \lambda t} \hat u$}\label{sec: inst 1}

\begin{proof}[Conclusion of the proof of Theorem \ref{thm: Z supcr 1}]
We point out that we make use of Theorem \ref{thm: gwp}, which will be proved in Section \ref{sec: gwp}. For every $s>0$, let $u_s := s \star \hat u$, and let $\psi_s$ be the solution to \eqref{com nls} with initial datum $u_s$. We have $u_s \to u$ as $s \to 0^+$, and hence it is sufficient to prove that $\psi_s$ blows-up in finite time. Let $t_{u_s,\mu}$ be defined by Lemma \ref{lem: prop fiber sub-sup}. Clearly $t_{u_s,\mu} = -s <0$, and by definition 
\[
E_\mu(u_s) = E_\mu(s \star \hat u) < E_\mu(t_{\hat u,\mu} \star \hat u) = \inf_{\cP_-^{a,\mu}} E_\mu.
\] 
Moreover, since $\hat \lambda<0$ and $\hat u \in H^1_{\rad}$, we have that $\hat u$ decays exponentially at infinity (see \cite{BerLio}), and hence $|x| u_s \in L^2(\R^N)$. Therefore, by Theorem \ref{thm: gwp} the solution $\psi_s$ blows-up in finite time .
\end{proof}

\subsection{Asymptotic behavior as $\mu \to 0^+$: proof of Theorem \ref{thm: mu to 0}}\label{sub: as 1}

In this subsection it is convenient to stress the dependence of $\tilde u$ and $\hat u$ on $\mu$, writing $\tilde u_\mu$ and $\hat u_\mu$. 
The value $a>0$ will always be fixed.

\begin{proof}[Proof of Theorem \ref{thm: mu to 0}: convergence of $\tilde u_\mu$]
For $a>0$ fixed, we know that $R_0(a,\mu) \to 0$ for $\mu \to 0^+$, and hence $|\nabla \tilde u_\mu|_2 < R_0(a,\mu) \to 0$ as well. Moreover, 
\[
0> m(a,\mu) = E_\mu(\tilde u_\mu) \ge \frac{1}{2} |\nabla \tilde u_\mu|_2^2 - \mu \frac{C_{N,q}^q}{q} a^{(1-\gamma_q) q} |\nabla \tilde u_\mu|_2^{\gamma_q q} -  \frac{C_{N,p}^p}{p} a^{(1-\gamma_p) p} |\nabla \tilde u_\mu|_2^{\gamma_p p} \to 0,
\]
which implies that $m(a,\mu) \to 0$.
\end{proof}

We consider now the behavior or $\hat u_\mu$. Before proceeding, we recall the properties of the unperturbed problem $\mu=0$ listed in Remark \ref{rem: mu = 0}.

\begin{lemma}\label{lem: sig infmax}
For any $\mu>0$ satisfying \eqref{cond sub sup} we have
\[
\sigma(a,\mu) = \inf_{u \in S_{a,r}} \max_{s \in \R} E_\mu(s \star u), \quad \text{and} \quad m(a,0) = \inf_{u \in S_{a,r}} \max_{s \in \R} E_0(s \star u).
\]
\end{lemma}
\begin{proof}
Recall that $\sigma(a,\mu) = \inf_{\cP_-^{a,\mu} \cap S_{a,r}} E_\mu =E_\mu(\hat u_\mu)$ (see Proposition \ref{prop: struct P}). Then, by Lemma \ref{lem: prop fiber sub-sup},
\[
\sigma(a,\mu) = E_\mu(\hat u_\mu) = \max_{s \in \R} E_\mu(s \star \hat u_\mu) \ge \inf_{u \in S_r} \max_{s \in \R} E_\mu(s \star u). 
\]
On the other hand, for any $u \in S_{a,r}$ we have $t_{u,\mu} \star u \in \cP_-^{a,\mu}$, and hence
\[
\max_{s \in \R} E_\mu(s \star u) = E_\mu(t_{u,\mu} \star u) \ge \sigma(a,\mu).
\]
The proof for $m(a,0)$ is analogue.
\end{proof}

\begin{lemma}\label{lem: hat u bdd}
For any $0<\mu_1<\mu_2$, with $\mu_2$ satisfying \eqref{cond sub sup}, it results that $\sigma(a,\mu_2) \le  \sigma(a,\mu_1) \le m(a,0)$.
\end{lemma}

\begin{proof}
By Lemma \ref{lem: sig infmax} \[
\sigma(a,\mu_2) \le \max_{s \in \R} E_{\mu_2}(s \star \hat u_{\mu_1}) \le \max_{s \in \R} E_{\mu_1}(s \star \hat u_{\mu_1}) = E_{\mu_1}(\hat u_{\mu_1}) = \sigma(a,\mu_1). 
\]
In the same way, we can also check that $\sigma(a,\mu_1) < m(a,0)$.
\end{proof}

\begin{proof}[Proof of Theorem \ref{thm: mu to 0}: convergence of $\hat u_\mu$]
The proof is similar to the one of Lemma \ref{lem: conv PS subcr}. Let us consider $\{\hat u_\mu: 0<\mu < \bar \mu\}$, with $\bar \mu$ small enough. At first, we show that $\{\hat u_\mu\}$ is bounded in $H^1$. This follows by Lemma \ref{lem: hat u bdd}, observing that, since $\hat u_\mu \in \cP_{a,\mu}$,
\begin{align*}
m(a,0) &\ge \sigma(a,\mu) = E_\mu(\hat u_\mu) \ge \frac{1}2\left( 1- \frac{2}{\gamma_p p}\right) |\nabla \hat u_\mu|_2^2 - \frac{\mu}{q} \left(1-\frac{\gamma_q q}{\gamma_p p}\right) |\hat u_\mu|_q^q \\
& \ge \frac{1}2\left( 1- \frac{2}{\gamma_p p}\right) |\nabla \hat u_\mu|_2^2 - \frac{\mu}{q} \left(1-\frac{\gamma_q q}{\gamma_p p}\right)C_{N,q}^q a^{(1-\gamma_q)q} | \nabla \hat u_\mu|_2^{\gamma_q q}.
\end{align*}
Since each $\hat u_\mu$ is a positive real-valued radial function in $S_a$, we deduce that up to a subsequence $\hat u_\mu \weak \hat u$ weakly in $H^1$, strongly in $L^p \cap L^q$ and a.e. in $\R^N$, as $\mu \to 0^+$\footnote{If $N=1$, we proceed in the same way observing that each $\hat u_\mu$ is also radially decreasing, see Remark \ref{rmk: sub sup 1}.}. Since $\hat u_\mu$ solves \eqref{stat com} for $\hat \lambda_\mu<0$, from $P_\mu(\hat u_\mu) = 0$ we infer that
\[
\hat \lambda_\mu a^2 = |\nabla \hat u_\mu|_2^2 - \mu |\hat u_\mu|_q^q -|\hat u_\mu|_p^p = (\gamma_q-1)\mu |\hat u_\mu|_q^q + (\gamma_p-1)|\hat u_\mu|_p^p,
\]
and hence also $\hat \lambda_\mu$ converges (up to a subsequence) to some $\hat \lambda \le 0$, with $\hat \lambda =0$ if and only if the weak limit $\hat u \equiv 0$. We claim that $\hat \lambda<0$. Indeed, by weak convergence $\hat u$ is a non-negative real radial solution to 
\begin{equation}\label{eq w lim}
-\Delta \hat u = \hat \lambda \hat u +  |\hat u|^{p-2} \hat u \qquad \text{in $\R^N$},
\end{equation}
and in particular by the Pohozaev identity $|\nabla \hat u|_2^2 = \gamma_p |\hat u|_p^p$. But then, using the boundedness of $\{\hat u_\mu\}$ and Lemma \ref{lem: hat u bdd}, we deduce that
\begin{align*}
E_0(\hat u) &= \frac{1}{p}\left(\frac{\gamma_p p}{2} -1 \right) |\hat u|_p^p = \lim_{\mu \to 0^+} \left[\frac{1}{p}\left(\frac{\gamma_p p}{2} -1 \right) |\hat u|_p^p + \frac{\mu}q\left( \frac{\gamma_q q}{2}-1\right) |\hat u_\mu|_q^q\right] \\
& = \lim_{\mu \to 0^+} E_\mu(\hat u_\mu) = \lim_{\mu \to 0^+} \sigma(a,\mu) \ge \sigma(a,\bar \mu)>0
\end{align*} 
which implies that $\hat u \not \equiv 0$, and in turn yields $\hat \lambda<0$. At this point, exactly as in Lemma \ref{lem: conv PS subcr} we deduce that $\hat u_\mu \to \hat u$ strongly in $H$. By regularity and the strong maximum principle, $\hat u \in S_a$ is a positive real radial solution to \eqref{eq w lim}, thus a ground state $\tilde u_0= \hat u$ of $E_0|_{S_a}$. Since the positive radial ground state is unique, it is not difficult to infer that the convergence $\hat u_\mu \to \tilde u_0$ takes place for the all family $\{\hat u_\mu\}$ (and not only for a subsequence). Moreover $\sigma(a,\mu) \to m(a,0)$.
\end{proof}

\subsection{Asymptotic behavior as $q \to \bar p^-$: proof of Theorem \ref{thm: q to bar p}}\label{sub: as 2}

In this subsection it is convenient to stress the dependence of $\tilde u$ on $q$, writing $\tilde u_q$.

\begin{proof}[Proof of Theorem \ref{thm: q to bar p}]
We recall that $|\nabla \tilde u_q|_2 < R_0 =R_0(q)$, where $R_0(q)$ is defined in Lemma \ref{lem: struct h}. The thesis follows then directly recalling that $R_0(q) < \bar t(q)$, with $\bar t= \bar t(q)$ defined in \eqref{def C_1} (see Remark \ref{rem: st R_0}). Indeed, passing to the limit as $q \to \bar p^-$ in \eqref{def C_1}, we deduce that 
\[
\bar t (q) =   \left( \frac{p(2-\gamma_q q)}{2C_{N,p}^p (\gamma_p p-\gamma_q q)}   \right)^{\frac{1}{\gamma_p p -2}} a^{-  \frac{(1-\gamma_p)p}{\gamma_p p -2}} \to 0,
\]
since $\gamma_q q \to 2^-$ for $q \to \bar p^-$.
\end{proof}

\section{Properties of ground states II}\label{sec:10}

In this section we prove Proposition \ref{prop: natural}, and Theorems \ref{thm: Z supcr 2} and \ref{thm: Z supcr 3}. We point out that we will use Theorem \ref{thm: gwp}, whose proof is contained in the next section. Once again, we omit the dependence on functionals and sets on $a$ and on $\mu$, which are assumed to be fixed. 

\begin{proof}[Proof of Proposition \ref{prop: natural}]
We recall that, under the assumptions of Theorems \ref{thm: subcr-supcr}, \ref{thm: supcr2} or \ref{thm: sup mu <0}, $\cP$ is a smooth manifold of codimension $2$ in $H$, and its subset $\cP_0$ is empty. If $u \in \cP$ is critical point for $E|_{\cP}$, then by the Lagrange multipliers rule there exists $\lambda, \nu \in \R$ such that
\[
dE(u)[\varphi] - \lambda \int_{\R^N} u \bar \varphi - \nu dP(u)[\varphi] = 0
\]
for every $\varphi \in H$, that is
\[
(1-2\nu) (-\Delta u) = \lambda u + (1-\nu \gamma_p p) |u|^{p-2} u + \mu(1-\nu \gamma_q q)|u|^{q-2} u \qquad \text{in $\R^N$}.
\]
We have to prove that $\nu=0$, and to this end we observe that by the Pohozaev identity
\[
(1-2\nu)|\nabla u|_2^2 = \mu \gamma_q(1-\nu \gamma_q q) |u|_q^q + \gamma_p(1-\nu \gamma_p p) |u|_p^p.
\]
Since $u \in \cP$, this implies that
\[
\nu \left(2|\nabla u|_2^2 - \mu q \gamma_q^2|u|_q^q - p \gamma_p^2|u|_p^p\right) = 0.
\]
But the term inside the bracket cannot be $0$, since $u \not \in \cP_0$, and then necessarily $\nu=0$. 
\end{proof}


\begin{proof}[Proof of Theorem \ref{thm: Z supcr 2}]
We start by describing the structure of the $Z$ of ground states. If $u \in Z$, then $u \in \cP$ and $E(u) = m(a,\mu) = \inf_{\cP} E$. We claim that
\begin{equation}\label{claim Z}
u \in Z \quad \implies \quad |u| \in Z, \quad |\nabla |u||_2 = |\nabla u|_2.
\end{equation}
To prove the claim, we observe that $E(|u|) \le E(u)$ and $P(|u|) \le P(u) = 0$. Then, by Lemma \ref{lem: fiber cr}, there exists $t_{|u|} \le 0$ with $t_{|u|} \star |u| \in \cP$, and by definition of $t_{|u|}$ we have
\begin{align*}
m(a,\mu) &\le E(t_{|u|} \star |u|) = e^{2 t_{|u|}} \left( \frac12\left(1- \frac{2}{\gamma_p p}\right) |\nabla |u||_2^2 - \frac{\mu}{\bar p}\left(1- \frac{2}{\gamma_p p}\right) |u|_{\bar p}^{\bar p}\right) \\
& \le e^{2 t_{|u|}} \left( \frac12\left(1- \frac{2}{\gamma_p p}\right) |\nabla u|_2^2 - \frac{\mu}{\bar p}\left(1- \frac{2}{\gamma_p p}\right) |u|_{\bar p}^{\bar p}\right)  = e^{2 t_{|u|}} E(u) = e^{2 t_{|u|}} m(a,\mu),
\end{align*}
where we used the fact that $u, t_{|u|} \star |u| \in \cP$, and $E(u) = m(a,\mu)$. Since $t_{|u|} \le 0$, we deduce that necessarily $t_{|u|} = 0$, that is $P(|u|)= 0$, and since also $P(u) = 0$ it follows that
\[
|u| \in \cP, \quad |\nabla |u||_2 = |\nabla u|_2, \quad \text{and} \quad E(|u|) = m(a,\mu).
\]
This proves claim \eqref{claim Z}.

Having shown that $|u|$ minimizes $E$ on $\cP$, we have that $|u|$ is a non-negative real valued solution to \eqref{stat com} for some $\lambda \in \R$, by Proposition \ref{prop: natural}. By regularity and the strong maximum principle, it is a $C^2$ positive solution. Using also the fact that $|\nabla |u||_2 = |\nabla u|_2$, we can then proceed as in \cite[Theorem 4.1]{HajStu}, obtaining the characterization of $Z$. 

We prove now that if $u \in Z$, then the associated Lagrange multiplier $\lambda$ is negative. This follows easily by testing \eqref{stat com} with $u$, and using the fact that $u \in \cP$:
\[
\lambda a^2 = |\nabla u|_2^2 -\mu|u|_q^q - |u|_p^p = \mu(\gamma_q -1)|u|_q^q +(\gamma_p -1) |u|_p^p<0,
\]
since $\gamma_q, \gamma_p <1$ by definition.

It remains to show that, if $u \in Z$, then the standing wave $e^{-i \lambda t} u(x)$ is strongly unstable. In light of Theorem \ref{thm: gwp}, and using the fact that $\lambda<0$, we can repeat word by word the argument in Subsection \ref{sec: inst 1}.
\end{proof}

The proof of Theorem \ref{thm: Z supcr 3} is analogue. The only difference stays in the verification of the fact that any Lagrange multiplier associated to a ground state is negative. To this end, we have to proceed as in the proof of Lemma \ref{lem: conv PS subcr}, step 3, using assumption \eqref{hp mu < 0 conv}. We omit the details.

\section{Global existence and finite time blow-up}\label{sec: gwp}

In this section we prove Theorem \ref{thm: gwp}, giving a unified proof for the three cases considered in the theorem. Notice that, in all of them, the existence and uniqueness of a unique global maximum point $t_{u,\mu}$ for $\Psi^\mu_u$ was already established in Lemmas \ref{lem: prop fiber sub-sup}, \ref{lem: fiber cr} and \ref{lem: fiber mu < 0}. Under the assumptions of Theorem \ref{thm: subcr-supcr}, we actually need a further preliminary result. Since $a$ and $\mu$ are fixed, we omit the dependence on these quantities from now on. 

\begin{lemma}\label{lem: 1710}
Under the assumptions of Theorem \ref{thm: subcr-supcr}, there exists $M>0$ such that $t_u<0$ for every $u \in S$ with $P(u) < - M$. 
\end{lemma}

\begin{proof}
By the Gagliardo-Nirenberg inequality
\[
P(u) \ge |\nabla u|_2^2 - \mu \gamma_q C_{N,q}^q a^{(1-\gamma_q)q} |\nabla u|_2^{\gamma_q q} - \gamma_p C_{N,p}^p a^{(1-\gamma_p) p} |\nabla u|_2^{\gamma_p p}.
\]
This means that $P(u) \ge g(|\nabla u|_2)$ for the function $g: (0,+\infty) \to \R$ defined by 
\[
g(t) = t^2 - \mu \gamma_q C_{N,q}^q a^{(1-\gamma_q)q} t^{\gamma_q q} - \gamma_p C_{N,p}^p a^{(1-\gamma_p) p} t^{\gamma_p p}.
\]
Proceeding as in Lemma \ref{lem: struct h}, and using assumption \eqref{cond sub sup}, it is not difficult to check that $g$ is positive on an interval $(R_2,R_3)$ with $R_2>0$. In particular, since $g(0^+) = 0^-$ and $g$ is continuous, there exists $M>0$ such that $g \ge -M$ on $[0, R_2]$. 

From Lemma \ref{lem: prop fiber sub-sup}, we know that $s_u$ is the lowest zero of $\Psi_u'$, and that $\Psi_u'<0$ for $s<s_u$. Since $\Psi_u'(\log (R_2/|\nabla u|_2)) \ge g(R_2) = 0$, necessarily $s_u < \log (R_2/|\nabla u|_2)$, that is $|\nabla (s_u \star u)|_2 \le R_2$, and hence 
\[
\inf_{s \in (-\infty, s_u]} \Psi_u'(s) = \inf_{s \in (-\infty, s_u]} P(s \star u)  \ge \inf_{s \in (-\infty, s_u]} g(|\nabla (s \star u)|_2) \ge \inf_{t \in [0,R_2]} g(t) =-M.
\]

Let us suppose now by contradiction that $P(u) < - M$ but $t_u \ge 0$. If $0 \in [s_u,t_u]$, then $P(u) = \Psi_u'(0) \ge 0$ (by monotonicity of $\Psi_u$, Lemma \ref{lem: prop fiber sub-sup}), which is not possible. Then $0<s_u$, but in this case 
\[
-M > P(u) = \Psi_u'(0) \ge \inf_{s \in (-\infty,s_u]} \Psi_u'(s) \ge -M,
\]
a contradiction again. 
\end{proof}

\begin{proof}[Global existence]
We assume that $t_u>0$ with $E(u) < \inf_{\cP_-}E$, and we show that the solution $\psi$ with initial datum $u$ is globally defined for positive times. For negative time we can use the same argument. By \cite[Proposition 3.1]{TaoVisZha}, the problem is locally well posed, $\psi \in C( (-T_{\min}, T_{\max}), H)$ for suitable $T_{\min}, T_{\max}>0$, and we have that either $T_{\max}=+\infty$, or $|\nabla \psi(t)|_2 \to +\infty$ as $t \to T_{\max}^-$. Thus, if by contradiction $T_{\max} <+\infty$, we have $|\nabla \psi(t)|_2 \to +\infty$ as $t \to T_{\max}^-$. Moreover, by the Gagliardo-Nirenberg inequality
\begin{equation}\label{E-P}
\begin{split}
E(\psi(t))- \frac{1}{\gamma_p p} P(\psi(t)) &=  \frac12\left(1-\frac{2}{\gamma_p p}\right) |\nabla \psi(t)|_2^2 - \frac{\mu}{q} \left(1- \frac{\gamma_q q}{\gamma_p p}\right)|\psi(t)|_q^q\\
& \ge  \frac12\left(1-\frac{2}{\gamma_p p}\right) |\nabla \psi(t)|_2^2 - \frac{1}{q} \left(1- \frac{\gamma_q q}{\gamma_p p}\right) C_{N,q}^q \mu a^{(1-\gamma_q q)} |\nabla \psi(t)|_2^{\gamma_q q}.
\end{split}
\end{equation}
We claim that this implies that
\begin{equation}\label{cl P infty}
E(\psi(t))- \frac{1}{\gamma_p p} P(\psi(t)) \to +\infty \qquad \text{as $t \to T_{\max}^-$}.
\end{equation}
Indeed, if the assumptions of Theorem \ref{thm: subcr-supcr} holds, then \eqref{cl P infty} follows from the fact that $\gamma_q q<2$; if the assumptions of Theorem \ref{thm: sup mu <0} holds, then \eqref{cl P infty} follows from the fact that $\mu<0$; and finally, if the assumptions of Theorem \ref{thm: supcr2} hold, then \eqref{cl P infty} follows by \begin{multline*}
\frac12\left(1-\frac{2}{\gamma_p p}\right) |\nabla \psi(t)|_2^2 - \frac{1}{q} \left(1- \frac{\gamma_q q}{\gamma_p p}\right) C_{N,q}^q \mu a^{(1-\gamma_q q)} |\nabla \psi(t)|_2^{\gamma_q q} \\
= \left(1-\frac{2}{\gamma_p p}\right) \left(\frac12-\frac{1}{\bar p} C_{N,\bar p}^{\bar p} \mu a^{4/N}  \right) |\nabla \psi(t)|_2^{2},
\end{multline*}
where the coefficient of $|\nabla \psi(t)|_2^2$ is positive.

Now, by conservation of energy \eqref{cl P infty} implies that $P(\psi(t)) \to -\infty$ as $t \to T_{\max}^+$; in particular, by Lemma \ref{lem: 1710}, Lemmas \ref{lem: fiber cr} and \ref{lem: fiber mu < 0} we have that $t_{\psi(T_{\max}-\eps)} <0$ if $\eps$ is small enough. But $t_{\psi(0)}>0$ by assumption, $u \mapsto t_u$ is continuous in $H$, and hence there exists $\tau  \in (0, T_{\max})$ such that $t_{\psi(\tau)} = 0$, namely $\psi(\tau) \in \cP_-$. Using again the conservation of the energy and the assumption on $E(u)$, we obtain
\[
\inf_{\cP_-}E> E(u) = E(\psi(\tau)) \ge \inf_{\cP_-}E,
\]
a contradiction.
\end{proof}

The proof of the finite time blow-up is inspired by the classical method of R. Glassey \cite{Gla}, refined by H. Berestycki and T. Cazenave \cite{BerCaz}.

\begin{proof}[Finite time blow-up]
For any $u \in S$, we define $\Phi_u: (0,+\infty) \to \R$ by $\Phi_u(s):= \Psi_u(\log s)$. Clearly, by Lemmas \ref{lem: prop fiber sub-sup}, \ref{lem: fiber cr} and \ref{lem: fiber mu < 0}, for every $u \in S$ the function $\Phi_u$ has a unique global maximum point $\tilde t_u= e^{t_u}$, and $\Phi_u$ is strictly decreasing and concave in $(\tilde t_u, +\infty)$\footnote{Since $\Phi_u'(s) = \Psi_u'(s)/s$, monotonicity properties of $\Phi_u$ can be inferred by those of $\Psi_u$. For the convexity and concavity, it is not difficult to modify the argument in Lemmas \ref{lem: prop fiber sub-sup}, \ref{lem: fiber cr} and \ref{lem: fiber mu < 0}.}. We claim that
\begin{equation}\label{cl inst}
\text{if $u \in S$ and $\tilde t_u \in (0,1)$, then }P(u) \le E(u) - \inf_{\cP_-} E.
\end{equation}
This follows from the concavity of $\Phi_u$ in $(\tilde t_u,+\infty)$, and from the fact that $\tilde t_u \in (0,1)$ (and hence $P(u)<0$, by monotonicity): indeed
\begin{align*}
E(u) &= \Phi_u(1) \ge \Phi_u(\tilde t_u) - \Phi_u'(1) (\tilde t_u -1)= E(t_u \star u) - |P (u)| (1-\tilde t_u) \\
& \ge \inf_{\cP_-} E - |P(u)| =  \inf_{\cP_-} E + P(u),
\end{align*}
which proves \eqref{cl inst}. 


Now, let us consider the solution $\psi$ with initial datum $u$. Since by assumption $t_u<0$, and the map $u \mapsto t_u$ is continuous, we deduce that $t_{\psi(\tau)}<0$ as well for every $|\tau|$ small, say $|\tau|< \bar \tau$. That is, $\tilde t_{\psi(\tau)} \in (0,1)$ for $|\tau|< \bar \tau$. By \eqref{cl inst} and recalling the assumption $E(u) < \inf_{\cP_-} E$, we deduce that 
\[
P(\psi(\tau)) \le E(\psi(\tau))- \inf_{\cP_-} E = E(u) - \inf_{\cP_-} E =: - \delta<0.
\]
for every such $\tau$, and hence $t_{\psi(\pm\bar \tau)} < 0$ (if at some instant $\tau \in (-\bar \tau, \bar \tau)$ we have $t_{\psi(\tau)}=0$, then $P(\psi(\tau))=0$, and this is not the case). By continuity again, the above argument yields
\begin{equation*}
P(\psi(t)) \le - \delta \qquad \text{for every $t \in (-T_{\min},T_{\max})$}.
\end{equation*}
%
%
%
%
%
%
To obtain a contradiction we recall that, since $|x| u \in L^2$ by assumption, by the virial identity \cite[Proposition 6.5.1]{Caz} the function
\[
f(t):= \int_{\R^N} |x|^2 |\psi(t,x)|^2\,dx
\]
is of class $C^2$, with $f''(t) = 8P(\psi(t)) \le -8 \delta$ for every $t \in (-T_{\min},T_{\max})$. Therefore
\[
0 \le f(t) \le - 4 \delta t^2  + f'(0) t + f(0)  \quad \text{for every $t \in (-T_{\min},T_{\max})$}.
\]
Since the right hand side becomes negative for $t$ large, this yields an upper bound on $T_{\max}$, which in turn implies final time blow-up. 
\end{proof}

\begin{proof}[Proof of Corollary \ref{cor: gwp}]
1) By Lemmas \ref{lem: prop fiber sub-sup}, \ref{lem: fiber cr}, \ref{lem: fiber mu < 0}, we have that $E(s \star u)  < \inf_{\cP_-} E$ for every $s < s_1$, with $s_1 \le t_u$ sufficiently ``small". Analogously, if $s > s_2$ with $s_2 \ge t_u$ large enough, then $E(s \star u) < \inf_{\cP_-} E$.\\
2) Assumption $P(u) >0$ reads $\Psi_u'(0)>0$. By the monotonicity of the fiber maps $\Psi_u$ in Lemmas \ref{lem: fiber cr} and \ref{lem: fiber mu < 0}, this directly implies $t_u>0$.\\
3) By Lemmas \ref{lem: prop fiber sub-sup}, \ref{lem: stima sup A cr} and \ref{lem: stima sup A mu < 0}, if $|\nabla u|_2$ is small enough, then necessarily $t_u>0$.\\
4)  By Lemmas \ref{lem: fiber cr} and \ref{lem: fiber mu < 0}, $P(u) = \Psi_u'(0) <0$ implies $t_{u}<0$. \\
5) If $t_u  \ge 0$, then by Lemma \ref{lem: prop fiber sub-sup}
\[
E(u) \ge \inf_{s \in (-\infty,t_u]} E(s \star u) = E(s_u \star u) \ge m(a,\mu).
\] 
Therefore, $E(u) < m(a,\mu)$ implies that $t_u<0$.
\end{proof}


\end{document}